\documentclass[10pt]{amsart}

\usepackage{amsfonts}
\usepackage{latexsym}
\usepackage{amssymb}
\usepackage{url}
\usepackage{mathrsfs}
\usepackage{tikz,tikz-qtree,ifthen}

\begin{document}

\newcommand{\bT}{\mathbb{T}}

\newcommand{\s}{\sigma}
\newcommand{\al}{\alpha}
\newcommand{\om}{\omega}
\newcommand{\be}{\beta}
\newcommand{\la}{\lambda}
\newcommand{\vp}{\varphi}

\newcommand{\bo}{\mathbf{0}}
\newcommand{\bone}{\mathbf{1}}

\newcommand{\sse}{\subseteq}
\newcommand{\contains}{\supseteq}
\newcommand{\forces}{\Vdash}

\newcommand{\FIN}{\mathrm{FIN}}
\newcommand{\Fin}{\mathrm{Fin}}
\newcommand{\fin}{\mathrm{fin}}

\newcommand{\ve}{\vee}
\newcommand{\w}{\wedge}
\newcommand{\bv}{\bigvee}
\newcommand{\bw}{\bigwedge}
\newcommand{\bcup}{\bigcup}
\newcommand{\bcap}{\bigcap}

\newcommand{\rgl}{\rangle}
\newcommand{\lgl}{\langle}
\newcommand{\lr}{\langle\ \rangle}
\newcommand{\re}{\restriction}

\newcommand{\bB}{\mathbb{B}}
\newcommand{\bP}{\mathbb{P}}
\newcommand{\bR}{\mathbb{R}}
\newcommand{\bW}{\mathbb{W}}
\newcommand{\bX}{\mathbb{X}}
\newcommand{\bN}{\mathbb{N}}
\newcommand{\bQ}{\mathbb{Q}}
\newcommand{\bS}{\mathbb{S}}
\newcommand{\St}{\tilde{S}}

\newcommand{\sd}{\triangle}
\newcommand{\cl}{\prec}
\newcommand{\cle}{\preccurlyeq}
\newcommand{\cg}{\succ}
\newcommand{\cge}{\succcurlyeq}
\newcommand{\dom}{\mathrm{dom}}
\newcommand{\ran}{\mathrm{ran}}

\newcommand{\lra}{\leftrightarrow}
\newcommand{\ra}{\rightarrow}
\newcommand{\llra}{\longleftrightarrow}
\newcommand{\Lla}{\Longleftarrow}
\newcommand{\Lra}{\Longrightarrow}
\newcommand{\Llra}{\Longleftrightarrow}
\newcommand{\rla}{\leftrightarrow}
\newcommand{\lora}{\longrightarrow}
\newcommand{\E}{\mathrm{E}}
\newcommand{\rank}{\mathrm{rank}}
\newcommand{\lefin}{\le_{\mathrm{fin}}}
\newcommand{\Ext}{\mathrm{Ext}}
\newcommand{\lelex}{\le_{\mathrm{lex}}}
\newcommand{\depth}{\mathrm{depth}}

\newcommand{\Erdos}{Erd{\H{o}}s}
\newcommand{\Pudlak}{Pudl{\'{a}}k}
\newcommand{\Rodl}{R{\"{o}}dl}
\newcommand{\Proml}{Pr{\"{o}}ml}
\newcommand{\Fraisse}{Fra{\"{i}}ss{\'{e}}}
\newcommand{\Sokic}{Soki{\'{c}}}
\newcommand{\Nesetril}{Ne{\v{s}}et{\v{r}}il}

\newtheorem{thm}{Theorem}
\newtheorem{prop}[thm]{Proposition}
\newtheorem{lem}[thm]{Lemma}
\newtheorem{cor}[thm]{Corollary}
\newtheorem{fact}[thm]{Fact}
\newtheorem{facts}[thm]{Facts}
\newtheorem*{Objective1}{Objective 1}
\newtheorem*{Objective2}{Objective 2}
\newtheorem*{Objective3}{Objective 3}
\newtheorem*{Objective4}{Objective 4}
\newtheorem*{thmMT}{Main Theorem}
\newtheorem*{thmMTUT}{Main Theorem for $\vec{\mathcal{U}}$-trees}
\newtheorem*{thmnonumber}{Theorem}
\newtheorem*{mainclaim}{Main Claim}
\newtheorem{claim}{Claim}
\newtheorem*{claim1}{Claim $1$}
\newtheorem*{claim2}{Claim $2$}
\newtheorem*{claim3}{Claim $3$}
\newtheorem*{claim4}{Claim $4$}

\theoremstyle{definition}
\newtheorem{defn}[thm]{Definition}
\newtheorem{example}[thm]{Example}
\newtheorem{conj}[thm]{Conjecture}
\newtheorem{prob}[thm]{Problem}
\newtheorem{examples}[thm]{Examples}
\newtheorem{question}[thm]{Question}
\newtheorem{problem}[thm]{Problem}
\newtheorem{openproblems}[thm]{Open Problems}
\newtheorem{openproblem}[thm]{Open Problem}
\newtheorem{conjecture}[thm]{Conjecture}
\newtheorem*{conjecture2}{Conjecture 2}
\newtheorem*{problem1}{Problem 1}
\newtheorem*{problem2}{Problem 2}
\newtheorem*{problem3}{Problem 3}
\newtheorem*{notation}{Notation}

\theoremstyle{remark}
\newtheorem*{rem}{Remark}
\newtheorem*{rems}{Remarks}
\newtheorem*{ack}{Acknowledgments}
\newtheorem*{note}{Note}
\newtheorem{claimn}{Claim}
\newtheorem{subclaim}{Subclaim}
\newtheorem*{subclaimnn}{Subclaim}
\newtheorem*{subclaim1}{Subclaim (i)}
\newtheorem*{subclaim2}{Subclaim (ii)}
\newtheorem*{subclaim3}{Subclaim (iii)}
\newtheorem*{subclaim4}{Subclaim (iv)}
\newtheorem{case}{Case}
 \newtheorem*{case1}{Case 1}
\newtheorem*{case2}{Case 2}
\newtheorem*{case3}{Case 3}
\newtheorem*{case4}{Case 4}

\newcommand{\noprint}[1]{\relax}
\newenvironment{nd}{\noindent\color{red}ND: }{}
\newenvironment{jsm}{\noindent\color{blue}SM: }{}
%To get rid of AB comments, replace {ab} below (not above) this line
%by {comment}\noprint1 and
%observe that the replacement is reversible.

\title[Topological Ramsey Spaces Dense in Forcings]{Topological Ramsey Spaces Dense in Forcings}


\author{Natasha Dobrinen}
\address{Department of Mathematics\\
  University of Denver \\
  C.M.\ Knudson Hall 302\\
   2290 S.\ York St. \\ Denver, CO \ 80208 U.S.A.}
\email{natasha.dobrinen@du.edu}
\urladdr{\url{http://web.cs.du.edu/~ndobrine}}
\thanks{This work  was partially supported by
 National Science Foundation Grant DMS-1600781}

\begin{abstract}
Topological Ramsey spaces are spaces which support infinite dimensional Ramsey theory  similarly to the Ellentuck space.
Each topological Ramsey space is endowed with a partial ordering which can be modified to a $\sigma$-closed  `almost reduction' relation analogously to the partial ordering of `mod finite'  on $[\om]^{\om}$.
Such forcings add new ultrafilters satisfying weak partition relations and have complete combinatorics.
In  cases where  a forcing turned out to be equivalent to a  topological Ramsey space, the
 strong Ramsey-theoretic techniques have
aided in
 a fine-tuned analysis of the Rudin-Keisler and Tukey structures associated with the forced ultrafilter and in discovering new ultrafilters with complete combinatorics.
 This expository paper
provides
 an overview of this collection of results  and an entry point for those interested  in
using topological Ramsey space techniques to gain finer insight into  ultrafilters satisfying weak partition relations.
\end{abstract}

%\begin{keyword}
%ultrafilter \sep Tukey \sep cofinal map\sep finitary %approximation maps \sep partition property
%\MSC Primary: 54D80 \sep 03E04; Secondary:  03E05
%\end{keyword}
%\end{frontmatter}

\maketitle

%******************************************************
%****************************************************
%***************************************************

\section{Overview}\label{sec.overview}

Topological Ramsey spaces are essentially topological spaces which support infinite dimensional Ramsey theory.
The prototype of all topological Ramsey spaces is the Ellentuck space.
This is the space of all infinite subsets of the natural numbers equipped with the Ellentuck topology, a refinement of the usual metric, or equivalently, product  topology.
In  this refined topology, every subset of the Ellentuck space which has the property of Baire is Ramsey.
This extends the usual Ramsey Theorem for pairs or triples, etc., of natural numbers  to infinite dimensions, meaning sets of infinite subsets of the natural numbers, with the additional requirement that the sets be definable in some sense.

Partially ordering the members of the Ellentuck space by almost inclusion yields a forcing which is equivalent to forcing with the Boolean algebra $\mathcal{P}(\om)/\fin$.
This forcing adds a Ramsey ultrafilter.
Ramsey ultrafilters have strong properties:
They are Rudin-Keisler minimal,  Tukey minimal, and have  complete combinatorics over $L(\mathcal{R})$,  in  the presence of large cardinals.

These important features
 are not unique to the Ellentuck space.
Rather, the same or analogous properties hold for a general class of spaces called topological Ramsey spaces.
The class of such spaces were defined by abstracting  the key properties from seminal spaces of Ellentuck, Carlson-Simpson, and Milliken's space of block sequences, and others.
Building on the work of Carlson and Simpson \cite{Carlson/Simpson90}, the first to form an abstract approach to such spaces,
Todorcevic presented a more streamlined set of axioms guaranteeing a space is a topological Ramsey space
 in \cite{TodorcevicBK10}.
This is the setting that we  work in.

Topological Ramsey spaces come equipped with a partial ordering.
This partial ordering can be modified to
a naturally defined $\sigma$-closed partial ordering of {\em almost reduction}, similarly to how the partial ordering of inclusion modulo finite  is defined from the partial ordering of inclusion.
This almost reduction ordering
was defined for abstract topological Ramsey spaces by Mijares in \cite{Mijares07}.
He showed that
forcing with a topological Ramsey space partially ordered by almost reduction  adds a new ultrafilter on the countable base set of first approximations.
Such ultrafilters inherit some weak partition relations from the fact that they were forced by a topological Ramsey space;
they behave like weak versions of a Ramsey ultrafilter.

When an ultrafilter is forced by a topological Ramsey space,
 one immediately has strong techniques at one's disposal.
The Abstract Ellentuck Theorem  serves to both streamline proofs  and helps to clarify what exactly is causing the particular properties of the forced ultrafilter.
The structure of the topological Ramsey space aids in several factors of the analysis of the behavior of the ultrafilter.
The following  are made possible by knowing that a  given forcing is equivalent to forcing with some topological Ramsey space.

\begin{enumerate}
\item
A simpler reading of the Ramsey degrees of the forced ultrafilter.
\item
Complete Combinatorics.
\item
Exact Tukey and Rudin-Keisler structures, as well as the structure of the Rudin-Keisler classes inside the Tukey classes.
\item
New canonical equivalence relations on fronts - extensions of the \Erdos-Rado and \Pudlak-\Rodl\ Theorems.
\item
Streamlines and simplifies proofs, and reveals the underlying structure responsible for the properties of the ultrafilters.
\end{enumerate}

This article focuses on studies
of ultrafilters satisfying  weak partition relations and which can be forced by some $\sigma$-closed partial orderings
in  \cite{Dobrinen/Todorcevic14}, \cite{Dobrinen/Todorcevic15}, \cite{Dobrinen/Mijares/Trujillo14}, \cite{DobrinenJSL15}, \cite{DobrinenJML16}, and other work.
These works concentrate on weakly Ramsey ultrafilters and a family of ultrafilters with increasingly weak partition properties due to Laflamme in \cite{Laflamme89};
p-points forced {\em $n$-square forcing} by Blass in \cite{Blass73} which have Rudin-Keisler structure below them a diamond shape;
the $k$-arrow, not $k+1$-arrow ultrafilters of Baumgartner and Taylor in \cite{Baumgartner/Taylor78}, as well as the arrow ultrafilters;
new classes of p-points with weak partition relations;
non-p-points forced by $\mathcal{P}(\om\times\om)/(\Fin\otimes\Fin)$ and the natural hierarchy of  forcings of increasing complexity, $\mathcal{P}(\om^\al)/\Fin^{\otimes\al}$.

It turned out that the original forcings adding these ultrafilters   actually contain dense subsets which form
 topological Ramsey spaces.
The Ramsey structure of these spaces aided greatly in the analysis of the properties of the forced ultrafilters.
In the process  some new classes of ultrafilters with weak partition properties were also produced.
Though there are many other classes of ultrafilters  not yet studied in this context, the fact that in all these cases  dense subsets of the forcings forming  topological Ramsey spaces were found signifies  a strong connection between ultrafilters satisfying some partition relations and topological Ramsey spaces.
Thus, we make the following conjecture.

\begin{conjecture}\label{conj}
Every ultrafilter which satisfies some  partition relation and is forced by some $\sigma$-closed forcing  is actually forced by some topological Ramsey space.
\end{conjecture}

While this is a strong conjecture, so far there is no evidence to the contrary, and it is a  motivating thesis for using topological Ramsey spaces to find a unifying framework for  ultrafilters satisfying some weak partition properties.

Finally, a note about attributions: 
We attribute work as stated in the papers quoted.

%%%%%%%%%%%%%%%%%%%%%%
%%%%%%%%%%%%%%%%%%%%%%
%%%%%%%%%%%%%%%%%%%%%%
%%%%%%%%%%%%%%%%%%%%%%

\section{A few basic definitions}

Most definitions used will appear as needed throughout this article.
In this section we define a few notions needed throughout.

\begin{defn}
A {\em filter} $\mathcal{F}$ on a countable base set $B$ is  a collection of subsets of $B$ which is closed under finite intersection and closed under superset.
An {\em ultrafilter} $\mathcal{U}$ on a countable base set $B$ is
a filter such that each subset of $B$ or its complement is in $\mathcal{U}$.
\end{defn}

We hold to the convention that all ultrafilters are proper ultrafilters;
thus $\emptyset$ is not a member of any ultrafilter.

\begin{defn}\label{.defnppt}
An ultrafilter  $\mathcal{U}$ on a countable base set $B$ is
\begin{enumerate}
\item
  {\em Ramsey} if for each $k,l\ge 1$ and each coloring $c:[B]^k\ra l$,
there is a member $U\in\mathcal{U}$ such that
$c\re [U]^k$ is constant.
\item
{\em selective}
if for each function $f:\om\ra\om$, there is a member $X\in\mathcal{U}$ such that  $f$ is either constant or one-to-one on $U$.
\item
{\em Mathias-selective}
if for each collection $\{U_s:s\in [\om]^{<\om}\}$ of members in $\mathcal{U}$,
there is an $X\in \mathcal{U}$ such that for each $s\in[\om]^{<\om}$ for which $\max(s)\in X$,
 $X\setminus (\max(s)+1)\sse U_s$.
\end{enumerate}
\end{defn}

The three  definitions above are equivalent.
Booth proved in \cite{Booth70} that (1) and (2) are equivalent, and
Mathias proved in \cite{Mathias77} that (1) and (3) are equivalent.

\begin{defn}
An ultrafilter $\mathcal{U}$ on a countable base set $B$ is a
{\em p-point}
if for each sequence $U_n$, $n<\om$, of members of $\mathcal{U}$, there is an $X$ in $\mathcal{U}$ such that for each $n<\om$,
$X\sse^* U_n$.
Equivalently, $\mathcal{U}$ is a {\em p-point}
if for each function $f:\om\ra\om$ there is a member $X$ in $\mathcal{U}$ such that $f$ is either constant or finite-to-one on $X$.

$\mathcal{U}$ is {\em rapid} if
for each strictly increasing function $f:\om\ra\om$,
there is a member $X\in\mathcal{U}$ such that for each $n<\om$,
$|X\cap f(n)|<n$.
\end{defn}

Using the function definition of p-point, it is clear that a selective implies p-point, which in turn implies rapid.

A different hierarchy of  ultrafilters may  be formed by weakening the Ramsey requirement to only require some bound on the number of colors appearing, rather than requiring homogeneity on a member in the ultrafilter.
The first of this type of weakening is a {\em weakly Ramsey} ultrafilter, which is an ultrafilter $\mathcal{U}$ such that for each  $l\ge 3$ and coloring $f:[\om]^2\ra l$,
there is a member $X\in\mathcal{U}$ such that
the restriction of $f$ to $[X]^2$ takes no more than $2$ colors.
The usual notation to denote this statement is
\begin{equation}
\mathcal{U}\ra(\mathcal{U})^2_{l,2}.
\end{equation}
This idea can be extended to any $k\ge 2$,
defining a {\em $k$-Ramsey} ultrafilter to be one such that for each  $l>k$ and $f:[\om]^2\ra l$,
there is a member $X\in\mathcal{U}$ such that
the restriction of $f$ to $[X]^2$ takes no more than $k$ colors.
This is denoted
\begin{equation}
\mathcal{U}\ra(\mathcal{U})^2_{l,k}.
\end{equation}
As we shall review later, Laflamme forced
a hierarchy of ultrafilters  $\mathcal{U}_k$, $k\ge 1$, such that
$\mathcal{U}_k$ is $k+1$-Ramsey but not $k$-Ramsey
\cite{Laflamme89}.

Next, we present one of the most useful ways of constructing new ultrafilters from old ones.

\begin{defn}[Fubini Product]\label{defn.Fub}
Let $\mathcal{U}$ and $\mathcal{V}_n$, $n<\om$, be ultrafilters on $\om$.
The {\em Fubini product} of $\mathcal{U}$ and $(\mathcal{V}_n)_{n<\om}$ is the ultrafilter $\lim_{n\ra\mathcal{U}}\mathcal{V}_n$
on base set $\om\times\om$ such that a set $A\sse\om\times\om$ is in  $\lim_{n\ra\mathcal{U}}\mathcal{V}_n$
if and only if
\begin{equation}
\{n<\om:\{j<\om:(n,j)\in A\}\in\mathcal{V}_n\}\in\mathcal{U}.
\end{equation}
\end{defn}
In other words, a subset $A\sse\om\times\om$ is in
$\lim_{n\ra\mathcal{U}}\mathcal{V}_n$ if and only if for $\mathcal{U}$ many $n$, the $n$-th fiber of $A$ is a member of $\mathcal{V}_n$.
If all $\mathcal{V}_n$ are equal to the same ultrafilter $\mathcal{V}$,
then the Fubini product is written as $\mathcal{U}\cdot\mathcal{V}$.

The Fubini product construction can be continued recursively any countable ordinal many times.
In particular, for any countable ordinal $\al$,
the $\al$-th Fubini iterate of $\mathcal{U}$ is denoted by  $\mathcal{U}^{\al}$.
The importance of this fact will be seen in later sections.

%%%%%%%%%%%%%%%%%%%%%%%%
%%%%%%%%%%%%%%%%%%%%%%%%
%%%%%%%%%%%%%%%%%%%%%%%%
%%%%%%%%%%%%%%%%%%%%%%%%

\section{The Prototype Example:  Ramsey Ultrafilters and the Ellentuck space}\label{sec.RamseyUF}

The connections and interactions between  Ramsey ultrafilters and the Ellentuck space provide the fundamental example  of the phenomena we are illustrating in this article.
Recall that  a {\em Ramsey ultrafilter} is an ultrafilter $\mathcal{U}$ on a countable base set,  usually taken to be $\om$, which contains witnesses of Ramsey's Theorem:
For each $k,l\ge 1$ and each coloring $c:[\om]^k\ra l$,
there is an $X\in\mathcal{U}$ such that $c\re [X]^k$ is constant.
Ramsey ultrafilters are forced by $\mathcal{P}(\om)/\Fin$, which is a $\sigma$-closed forcing.
This is forcing equivalent to the partial ordering $([\om]^{\om},\sse^*)$,
where for $X,Y\in[\om]^{\om}$, $Y\sse^* X$ if and only if $Y\setminus X$ is finite.
This section reviews the complete combinatorics of Ramsey ultrafilters,  the exact Rudin-Keisler and Tukey structures connected with Ramsey ultrafilters,
 and the roles played by the Ellentuck space, either implicitly or explicitly in these results,
and the crucial theorems of     Nash-Williams, Ellentuck, and \Pudlak-\Rodl.
This
provides  the groundwork from which to understand the more general results.

\subsection{Complete combinatorics of Ramsey ultrafilters}\label{subsection.completecombRamsey}

Saying that an ultrafilter has {\em  complete combinatorics}
means  that   there is some  forcing and some well-defined  combinatorial property such that any ultrafilter satisfying that property is generic for the forcing over some well-defined inner model.
There are two main formulations of complete combinatorics for Ramsey ultrafilters.
The first has its inception  in  work of Mathias in
  \cite{Mathias77} and was formulated by Blass in \cite{Blass88}:
Any Ramsey ultrafilter in the model $V[G]$ obtained by L\'{e}vy collapsing a Mahlo cardinal to $\aleph_1$ is $\mathcal{P}(\om)/\Fin$-generic over HOD$(\mathbb{R})^{V[G]}$.
Thus, we say that Ramsey ultrafilters have {\em complete combinatorics} over HOD$(\mathbb{R})^{V[G]}$, where
$V[G]$ is obtained by L\'{e}vy collapsing a Mahlo cardinal to $\aleph_1$.
This form of complete combinatorics does not take place in the original model $V$, but only presupposes the existence of a Mahlo cardinal.

The second formulation  of complete combinatorics is due to Todorcevic (see Theorem 4.4 in \cite{Farah98}) building on  work of Shelah and  Woodin  \cite{Shelah/Woodin90}.
It presupposes the existence of large cardinals stronger than a Mahlo but has the advantage that the statement is with respect to the canonical inner model $L(\mathbb{R})$ inside $V$ rather than in a forcing extension of $V$ collapsing a Mahlo cardinal.
If $V$ has a supercompact cardinal (or somewhat less),
then any Ramsey ultrafilter in $V$ is generic for the forcing $\mathcal{P}(\om)/\Fin$ over the Solovay model $L(\mathbb{R})$ inside $V$.
Thus, we say that in the presence of certain large cardinals, each Ramsey ultrafilter in $V$ has {\em complete combinatorics} over $L(\mathbb{R})$.
This second formulation lends itself to natural generalizations to forcing with abstract topological Ramsey spaces, as we shall review later.

\subsection{Rudin-Keisler order}\label{subsection.RK}

The well-studied Rudin-Keisler order on ultrafilters is a quasi-ordering in which `stronger' ultrafilters are smaller.
Given ultrafilters $\mathcal{U}$ and $\mathcal{V}$,
we say that $\mathcal{V}$ is {\em Rudin-Keisler reducible to $\mathcal{U}$}
if there is a function $f:\om\ra\om$ such that
\begin{equation}
\mathcal{V}=f(\mathcal{U}):=\{X\sse\om:f^{-1}(X)\in\mathcal{U}\}.
\end{equation}

Two ultrafilters are Rudin-Keisler (RK) equivalent if and only if each is RK-reducible to the other.
In this case, we write $\mathcal{U}\equiv_{RK}\mathcal{V}$.
It turns out  that  two ultrafilters are RK equivalent if and only if there is a bijection between their bases  taking one ultrafilter to the other (see  \cite{BlassThesis} or \cite{Booth70}).
Thus, we shall use the terminology {\em RK equivalent} and {\em isomorphic} interchangeably.
The collection of all ultrafilters RK equivalent to a given ultrafilter $\mathcal{U}$ is called the {\em RK class} or {\em isomorphism class} of $\mathcal{U}$.

Recall that an ultrafilter $\mathcal{U}$ is {\em selective}
if for each function $f:\om\ra\om$,
there is a member $X\in\mathcal{U}$ such that $f$ is either one-to-one or constant on $X$.
If $f$ is constant on some member of $\mathcal{U}$,
then  the ultrafilter $f(\mathcal{U})$ is principal.
If $f$ is one-to-one on some member of $\mathcal{U}$, then $f(\mathcal{U})$ is isomorphic to $\mathcal{U}$.
Hence, selective  ultrafilters are Rudin-Keisler minimal among nonprincipal ultrafilters.

\subsection{Tukey order on ultrafilters}\label{subsection.Tukey}
The Tukey order between partial orderings
was defined by Tukey in order to study convergence in Moore-Smith topology.
In recent decades it has found deep applications in areas where isomorphism is too fine a notion to reveal  useful information.
In the setting of ultrafilters partially ordered by reverse inclusion,
the Tukey order is a coarsening of the Rudin-Keisler order and provides information about the cofinal types of  ultrafilters.

Given ultrafilters $\mathcal{U}$ and $\mathcal{V}$,
we say that $\mathcal{V}$ is {\em Tukey reducible to $\mathcal{U}$}
if there is a function $f:\mathcal{U}\ra\mathcal{V}$ such that for each filter base $\mathcal{B}$ for $\mathcal{U}$,
$f''\mathcal{B}$ is a filter base for $\mathcal{V}$.
Such a map is called a {\em cofinal} map or a {\em convergent} map.
Equivalently, $\mathcal{V}$ is Tukey reducible to $\mathcal{U}$ if there is a function $g:\mathcal{V}\ra\mathcal{U}$ such that for each unbounded subset $\mathcal{X}\sse\mathcal{V}$,
the image $g''\mathcal{X}$ is unbounded in $\mathcal{U}$.
It is worth noting that whenever $\mathcal{V}\le_T\mathcal{U}$,
then there is a {\em monotone} cofinal map witnessing this; that is,
a map  $f:\mathcal{U}\ra\mathcal{V}$ such  that whenever $X\contains Y$ are members of $\mathcal{U}$,
then $f(X)\contains f(Y)$.

If both $\mathcal{U}$ and $\mathcal{V}$ are Tukey reducible to each other, then we say that they are {\em Tukey equivalent}.
For directed partial orderings, Tukey equivalence is the same as cofinal equivalence:
There is some other directed partial ordering into which they both embed as cofinal subsets.
Since for any ultrafilter $\mathcal{U}$, the partial ordering   $(\mathcal{U},\contains)$ is directed,
 Tukey equivalence between ultrafilters is the same as cofinal equivalence.
The collection of all ultrafilters Tukey equivalent to $\mathcal{U}$ is called the {\em Tukey type} or {\em cofinal type} of $\mathcal{U}$.

Each Rudin-Keisler map induces a monotone cofinal map.
If $h:\om\ra\om$ and $\mathcal{V}=h(\mathcal{U})$,
then the map $f:\mathcal{U}\ra\mathcal{V}$ given by $f(X)=\{h(n):n\in X\}$, for $X\in\mathcal{U}$, is a cofinal map witnessing that $\mathcal{V}\le_T\mathcal{U}$.
Thus, Tukey types form a coarsening of the RK classes of ultrafilters.

Todorcevic proved in \cite{Raghavan/Todorcevic12} that, analogously to the Rudin-Keisler order,
Ramsey ultrafilters are minimal among nonprincipal ultrafilters  in the Tukey ordering.
His  proof uses a theorem that p-points carry continuous cofinal maps \ and a theorem of \Pudlak\ and \Rodl\ regarding canonical equivalence relations on barriers of the Ellentuck space.
These will be discussed below, after which we will return to  an outline of the proof of this theorem.

\subsection{Continuous cofinal maps from p-points}\label{subsection.ctscof}

It was proved in \cite{Dobrinen/Todorcevic11} that every p-point carries continuous cofinal maps.
The members of a given ultrafilter $\mathcal{U}$ are subsets of $\om$.
Using the natural correspondence between a subset of $\om$ and its characteristic function as an infinite sequence of $0$'s and $1$'s,
each ultrafilter on $\om$ may be seen as a subspace of the Cantor space $2^{\om}$, endowed with the product topology.
A continuous map from $2^{\om}$ into itself is a function such that the preimage of any open set is open.
This amounts to continuous functions having initial segments of their images being decided by finite amounts of information.
Thus,  $f:2^{\om}\ra 2^{\om}$ is continuous if and only if there is a finitary function $\hat{f}:2^{<\om}\ra 2^{<\om}$ such that
$\hat{f}$ preserves end-extensions and reproduces $f$.
Precisely,
for $s\sqsubseteq t$, $\hat{f}(s)\sqsubseteq \hat{f}(t)$,
and for each $X\in 2^{\om}$,
$f(X)=\bigcup_{n<\om}\hat{f}(X\re n)$.
(For finite sets of natural numbers,
the notation $s\sqsubseteq t$ is used to denote  that $s$ is an initial segment of $t$, meaning that $s=\{n\in t:n\le \max(s)\}$.
$s\sqsubset t$ denotes that $s$ is a proper initial segment of $t$: $s$ is an initial segment of $t$ and $s$ is not equal to $t$.)

Recall that an ultrafilter $\mathcal{U}$ is a {\em  p-point} if whenever $X_n$, $n<\om$, are members of $\mathcal{U}$ such that each $X_{n+1}\sse^* X_n$,
then there is a member $U\in\mathcal{U}$ such that for each $n<\om$, $U\sse^* X_n$.
Such a set $U$ is called a {\em pseudointersection} of the sequence of $\{U_n:n<\om\}$.

\begin{thm}[Dobrinen/Todorcevic, \cite{Dobrinen/Todorcevic11}]\label{thm.20}
For each p-point $\mathcal{U}$,
if $f:\mathcal{U}\ra\mathcal{V}$ is a monotone cofinal map,
then there is a member $X\in\mathcal{U}$ such that $f$ is continuous when restricted to  the set $\mathcal{U}\re X:=\{U\in\mathcal{U}:U\sse X\}$.
Moreover,
there is a  monotone continuous function $f^*$ from $\mathcal{P}(\om)$ into $\mathcal{P}(\om)$
such that $f^*\re(\mathcal{U}\re\tilde{X})
= f\re(\mathcal{U}\re\tilde{X})$.
\end{thm}

%%%%%%%%%%%%%
%%%%%%%%%%%%%

\subsection{The Ellentuck Space}

The Ellentuck space has as its points the infinite subsets of the natural numbers, $[\om]^{\om}$.
For  $a\in[\om]^{<\om}$ and $X\in[\om]^{\om}$,
$a\sqsubset X$ denotes that $a=\{n\in X:n\le\max(a)\}$.
The basic open sets inducing the  Ellentuck topology  are of the following form:
Given a finite set $a\in [\om]^{<\om}$ and an infinite set $X\in[\om]^{\om}$,
define
\begin{equation}
[a,X]=\{Y\in[\om]^{\om}:a\sqsubset Y\sse X\}.
\end{equation}
The {\em Ellentuck topology} is the topology on the space $[\om]^{\om}$ induced by all basic open sets of the form $[a,X]$, for $a\in[\om]^{<\om}$ and $X\in[\om]^{\om}$.
Notice that this topology refines the usual metric or equivalently product topology on $[\om]^{\om}$.

It is this topology which is the correct one in which to understand infinite dimensional Ramsey theory.
Infinite dimensional Ramsey theory is the extension of Ramsey theory from finite dimensions, that is, colorings of $[\om]^k$ where $k$ is some positive integer, to colorings of $[\om]^{\om}$.
Assuming the Axiom of Choice,  the following statement is false: ``Given a function $f:[\om]^{\om}\ra 2$, there is an $M\in[\om]^{\om}$ such that $f$ is constant on $[M]^{\om}$."
However,
if the coloring is sufficiently definable, then Ramsey theorems hold.
This is the content of the progression from the Nash-Williams Theorem  \cite{NashWilliams65} through the work of  Galvin-Prikry \cite{Galvin/Prikry73}, Mathias \cite{Mathias77}, Silver \cite{Silver70} and Louveau \cite{Louveau74}
up to the theorem of Ellentuck showing that the Ellentuck topology is the correct topology in which to obtain optimal infinite dimensional Ramsey theory.

\begin{thm}[Ellentuck, \cite{Ellentuck74}]\label{thm.Ellentuck}
If $\mathcal{X}\sse[\om]^{\om}$ has the property of Baire in the Ellentuck topology,
then for any basic open set $[a,X]$,
there is a member $Y\in [a,X]$ such that
either $[a,Y]\sse\mathcal{X}$ or else
$[a,Y]\cap\mathcal{X}=\emptyset$.
\end{thm}

Those familiar with Mathias forcing will notice the strong correlation between Ellentuck's basic open sets and conditions in Mathias forcing.
Forcing with the collection of basic open  sets in the Ellentuck topology, partially ordered by inclusion,
is in fact equivalent to Mathias forcing.

\subsection{Fronts, barriers, and the Nash-Williams Theorem}\label{subsection.NW}

Fronts and barriers are collections of finite sets which approximate all infinite sets and are minimal in some sense.

\begin{defn}\label{defn.front}
A set $\mathcal{F}\sse[\om]^{<\om}$ is a {\em front} if
\begin{enumerate}
\item
For each $X\in[\om]^{\om}$, there is an $a\in\mathcal{F}$ such that $a\sqsubset X$.
\item
Whenever $a,b\in \mathcal{F}$ and $a\ne b$, then $a\not\sqsubset b$.
\end{enumerate}
$\mathcal{F}$ is a {\em barrier} if it satisfies (1) and also $(2')$ holds:
\begin{enumerate}
\item[($2'$)]
Whenever $a,b\in\mathcal{F}$ and $a\ne b$,
then $a\not\subset b$.
\end{enumerate}
\end{defn}

A family satisfying (2) is called {\em Nash-Williams} and a family satisfying ($2'$) is called {\em Sperner}.
The notions of front and barrier may be relativized to any infinite subset of $\om$.
By a theorem of Galvin,
for each front $\mathcal{F}$ on some infinite $M\sse \om$, there is an infinite subset $N\sse M$ such that  $\mathcal{F}|N:=\{a\in \mathcal{F}:a\sse N\}$ is a barrier on $N$.

Notice that for each $k<\om$, the set $[\om]^k$ is both a front and a barrier on $\om$.
The set $[\om]^k$ is in fact the uniform barrier of rank $k$.
Uniform barriers, and fronts, are defined by recursion on the rank.
Given uniform barriers $\mathcal{B}_n$ on $\om\setminus (n+1)$,
with the rank of $\mathcal{B}_n$ being $\al_n$, where either all $\al_n$ are the same or else they are strictly increasing,
the barrier
\begin{equation}\label{eq.barrierinduction}
\mathcal{B}=\{\{n\}\cup a:a\in\mathcal{B}_n,\ n<\om\}
\end{equation}
is a uniform barrier of rank $\sup\{\al_n+1:n<\om\}$.
The {\em Shreier barrier} $\mathcal{S}$ is the fundamental example of a uniform barrier of rank $\om$.
\begin{equation}
\mathcal{S}=\{a\in[\om]^{<\om}:|a|=\min(a)+1\}.
\end{equation}
This is the same as letting $\mathcal{B}_n=[\om\setminus (n+1)]^n$  and
defining $\mathcal{S}$ as in Equation (\ref{eq.barrierinduction}).

The Nash-Williams Theorem shows that every clopen subset of the Baire space with the metric topology  has the Ramsey property.
Though this follows from the Ellentuck Theorem, we state it here since it will be useful in several proofs which do not require the full strength of the Ellentuck Theorem.

\begin{thm}[Nash-Williams]\label{thm.NW}
Given any front $\mathcal{F}$ on an infinite set $M\sse\om$ and any partition of $\mathcal{F}$ into finitely many  pieces, $\mathcal{F}_i$, $i<l$ for some $l\ge 1$,
there is an infinite  $N\sse M$ such that
$\mathcal{F}|N\sse \mathcal{F}_i$
for one $i<l$.
\end{thm}

\begin{rem}
Any ultrafilter $\mathcal{U}$  generic for $([\om]^{\om},\sse^*)$
satisfies a generic version of the Nash-Williams Theorem.
By a density  argument,
for each $k\ge 1$, given any partition of $[\om]^k$  into finitely many sets,
there is a member $X\in\mathcal{U}$ such that $[X]^k$ is contained in one piece of the partition.
Thus, we may see the Ramsey property for the generic ultrafilter as  instances of the Nash-Williams Theorem for uniform barriers of finite rank.
\end{rem}

\subsection{Canonical equivalence relations on barriers}\label{subsection.canonicaleqrel}

Letting the number of colors increase,
we see that coloring with  infinitely many colors  corresponds  to  forming an equivalence relation on a collection of finite sets, where two finite sets are equivalent if and only if they have the same color.
The first Ramsey-like theorem for infinite colorings of finite sets of natural numbers is due to  \Erdos\ and Rado.
Though it is not always possible to find a large homogeneous set in one color,
it is possible to find a large set on which the equivalence relation is in some sense {\em canonical}.
This term is used to refer to
some simple equivalence relation which  once achieved, is inherited by all further infinite subsets.

Given $k\ge 1$,
the canonical equivalence relations on $[\om]^k$ are of the form $\E_I$, where $I\sse k$, defined by
\begin{equation}
a\, \E_I\, b \llra \{a_i:i\in I\}=\{b_i:i\in I\},
\end{equation}
where $a,b\in[\om]^k$ and $\{a_0,\dots,a_{k-1}\}$ and $\{b_0,\dots,b_{k-1}\}$ are the strictly increasing enumerations of $a$ and $b$.

\begin{thm}[\Erdos-Rado]\label{thm.ER}
Given $k\ge 1$ and an equivalence relation $\E$ on $[\om]^k$,
there is some infinite $M\sse \om$ and some $I\sse k$ such that
the restriction of $\E$ to $[M]^k$ is exactly $\E_I\re [M]^k$.
\end{thm}

It is often quite useful to think of canonical equivalence relations in terms of projection maps.
For $I\sse k$ and $a\in[\om]^k$,
let $\pi_I(a)=\{a_i:i\in I\}$.
Then $a\, \E_I\, b$ if and only if $\pi_I(a)=\pi_I(b)$.

The notion of canonical equivalence relations may be extended to all uniform barriers of any countable ordinal rank.
This is the content of the next theorem of \Pudlak\ and \Rodl.
Since the lengths of members of a uniform barrier of infinite rank are not bounded,
the notion of canonical becomes a bit less obvious at the start.
However, just as the canonical equivalence relations on $[\om]^k$ can be thought of as projections to the $i$-th members of $a$, for $i$ in a given indext set,
so too it can be instructive to think of irreducible functions as projections to certain indexed members of a given finite set.

\begin{defn}[irreducible function]\label{def.irred}
Let $\mathcal{B}$ be a barrier on $\om$.
A function $\varphi:\mathcal{B}\ra[\om]^{<\om}$
is {\em irreducible}
if
\begin{enumerate}
\item
For each $a\in\mathcal{B}$, $\varphi(a)\sse a$;
\item
For all $a,b\in\mathcal{B}$, if $\varphi(a)\ne\varphi(b)$,
then $\varphi(a)\not\sse\varphi(b)$.
\end{enumerate}
\end{defn}

Property (2)  implies that the image set $\{\varphi(a):a\in\mathcal{B}\}$ is a Sperner set, and thus forms a barrier on some infinite set.

\begin{thm}[\Pudlak-\Rodl]\label{thm.PR}
Given an equivalence relation $\E$ on a barrier $\mathcal{B}$ on $\om$,
there is an infinite subset $M\sse \om$ and  an irreducible function $\varphi$ which canonizes $\E$ on $\mathcal{B}\re M$:
For all $a,b\in\mathcal{B}|M$, $a\,\E\, b$
if and only if $\varphi(a)=\varphi(b)$.
\end{thm}

The \Erdos-Rado Theorem follows from the \Pudlak-\Rodl\ Theorem, and it may be instructive for the reader to prove this.

\begin{rem}
We stated that it can be instructive to think of irreducible functions as projection maps.
This view provides intuition for understanding irreducible maps canonizing equivalence relations on barriers for  topological Ramsey spaces which are
 more complex than the Ellentuck space.
\end{rem}

\subsection{Fubini iterates of ultrafilters and the correspondence with uniform fronts}

New ultrafilters may be constructed from given ultrafilters using the process of Fubini product and iterating it countably many times.
Recall that
the {\em Fubini product} $\lim_{n\ra \mathcal{U}}\mathcal{V}_n$ of $\mathcal{U}$ with $\mathcal{V}_n$, $n<\om$, is
\begin{equation}
\{A\sse \om\times\om: \{n<\om:\{j<\om:(n,j)\in A\}\in\mathcal{V}_n\}\in\mathcal{U}\}.
\end{equation}
Note that $[\om]^2$ is in one-to-one correspondence with  the upper triangle of $\om\times\om$.
As long as $\mathcal{U}$-many of the  ultrafilters $\mathcal{V}_n$ are nonprincipal,
the upper triangle $\{(n,j):n<j<\om\}$ is a member of $\lim_{n\ra\mathcal{U}}\mathcal{V}_n$.
Thus, $\lim_{n\ra\mathcal{U}}\mathcal{V}_n$ is isomorphic to the ultrafilter  $\mathcal{W}$ defined on base set $[\om]^2$
by
$B\sse[\om]^2\in\mathcal{W}$ if and only if
for $\mathcal{U}$ many $n<\om$, the set $\{j>n: \{n,j\}\in B\}\in\mathcal{V}_n$.
Hence
we may assume that the base set for the ultrafilter is $[\om]^2$ rather than $\om\times\om$.
It thus is natural to  let $\lim_{n\ra \mathcal{U}}\mathcal{V}_n$ denote this ultrafilter $\mathcal{W}$ on base set $[\om]^2$.

This connection continues when we iterate the Fubini product construction.
If we have ultrafilters $\mathcal{W}_n$ each of which is a Fubini product of some ultrafilters
and $\mathcal{X}$ is another ultrafilter,
then $\lim_{n\ra\mathcal{X}}\mathcal{W}_n$  is an ultrafilter on base set $\om\times\om\times\om$.
As long as the ultrafilters are nonprincipal, this
is isomorphic to an ultrafilter on base set $[\om]^3$,
since each ultrafilter $\mathcal{W}_n$ is (modulo some set in $\mathcal{W}_n$)  an ultrafilter on base set $[\om]^2$.

This recursive construction continues onward so that to each ultrafilter $\mathcal{W}$ which is obtained via a countable iteration of Fubini products of ultrafilters, there corresponds a barrier $\mathcal{B}$ of the same rank as the recursive rank of the construction of $\mathcal{W}$ so that the base set of $\mathcal{W}$ may without any loss of information be assumed to be $\mathcal{B}$.
This idea of using uniform barriers as the base sets for iterated Fubini products is due to Todorcevic.
It is delineated in more detail in \cite{DobrinenCanonicalMaps15}.

%%%%%%%%%%%%%%%%%%%%
%%%%%%%%%%%%%%%%%%%%
%%%%%%%%%%%%%%%%%%%%

\subsection{Ramsey ultrafilters are Tukey minimal, and the RK structure inside its Tukey type is exactly the Fubini powers of the Ramsey ultrafilter}\label{subsection.initialTRKstructure}

Given an ultrafilter $\mathcal{U}$ and $1\le\al<\om_1$,
 $\mathcal{U}^{\al}$ denotes the $\al$-th Fubini power of $\mathcal{U}$.
This is formed by the Fubini product where all the $\mathcal{V}_n=\mathcal{U}$ at each iteration of the Fubini product.
For limit $\al$, given $\mathcal{U}^{\beta}$ for all $\beta<\al$,
$\mathcal{U}^{\al}$ denotes $\lim_{n\ra\mathcal{U}}\mathcal{U}^{\beta_n}$, where $(\beta_n)_{n<\om}$  is any increasing sequence cofinal in $\al$.
The following theorem
is attributed to Todorcevic 
 in \cite{Raghavan/Todorcevic12}.

\begin{thm}[Todorcevic,  \cite{Raghavan/Todorcevic12}]
Each Ramsey ultrafilter is Tukey minimal.
Moreover, if $\mathcal{U}$ is Ramsey and $\mathcal{V}\le_T\mathcal{U}$,
then $\mathcal{V}\equiv_{RK}\mathcal{U}^{\al}$ for some $\al<\om_1$.
\end{thm}

The structure of his proof is as follows:
Suppose that $\mathcal{U}$ is a Ramsey ultrafilter, $\mathcal{V}\le_T\mathcal{U}$, and
$f:\mathcal{U}\ra\mathcal{V}$  is a
 monotone cofinal map.
\begin{enumerate}
\item
By Theorem \ref{thm.20},
there is a member $X\in\mathcal{U}$ such that $f$ is continuous when restricted to $\mathcal{U}\re X:=\{U\in\mathcal{U}:U\sse X\}$.
Without loss of generality, $f$ may be assumed to be continuous on all of $\mathcal{U}$.
\item
The finitary map $\hat{f}:[\om]^{<\om}\ra[\om]^{<\om}$ approximating $f$ on $\mathcal{U}$
is used to define the set $\mathcal{F}$ of all $U\cap n$ where $U\in\mathcal{U}$ and $n$ is minimal such that $\hat{f}(U\cap n)\ne\emptyset$.
The set $\mathcal{F}$ forms a {\em front on $\mathcal{U}$} since each member of $\mathcal{U}$ end-extends some member of $\mathcal{F}$, and the minimality of the members of $\mathcal{F}$ imply that this set is Nash-Williams.
\item
The set $\mathcal{F}$ is the countable base set for a new ultrafilter.
For each set $U\in\mathcal{U}$,
\begin{equation}
 \mathcal{F}|U= \{a\in\mathcal{F}:a\sse U\}.
\end{equation}
$\mathcal{U}\re\mathcal{F}$ denotes the filter on base set $\mathcal{F}$ generated by the sets $\mathcal{F}|U$, for $U\in\mathcal{U}$.
This filter turns out to be an ultrafilter on base set $\mathcal{F}$.
This follows since $\mathcal{U}$ being Ramsey implies, with a bit of work, that
for any partition of $\mathcal{F}$ into two pieces $\mathcal{F}_0$, $\mathcal{F}_1$,
there is an $X\in\mathcal{U}$ such that  $\mathcal{F}|X\sse\mathcal{F}_i$ for one $i<2$.

\item
Define a function $g:\mathcal{F}\ra \om$ by
$g(a)=\min(\hat{f}(a))$.
Then $g$ is a Rudin-Keisler map from the countable base set $\mathcal{F}$ to $\om$.
It turns out that the ultrafilter $\mathcal{V}$ is equal to the Rudin-Keisler image $g(\mathcal{U}\re \mathcal{F})$.
Thus, to compare $\mathcal{V}$ with $\mathcal{U}$, it is sufficient to compare $g(\mathcal{U}\re\mathcal{F})$ with $\mathcal{U}$.
\item
$g$ colors $\mathcal{F}$ into infinitely many colors, thus inducing
 an equivalence relation on $\mathcal{F}$:
For $a,b\in\mathcal{F}$, $a$ and $b$ are equivalent if and only if $g(a)=g(b)$.
By the \Pudlak-\Rodl\ Theorem applied to this equivalence relation on  $\mathcal{F}$,
 there is a member $U\in\mathcal{U}$
such that below $U$, the equivalence relation is canonical witnessed by some irreducible map $\varphi$.
The image of $\mathcal{F}$ under $\varphi$ also is a barrier (possibly restricting below some smaller set in $\mathcal{U}$); let $\mathcal{B}=\{\varphi(a):a\in\mathcal{F}\}$.
Thus,
\begin{equation}
\mathcal{V}=g(\mathcal{U}\re \mathcal{F})\cong \varphi(\mathcal{U}\re \mathcal{F})
=\mathcal{U}\re\mathcal{B},
\end{equation}
where $\varphi(\mathcal{U}\re\mathcal{F})$ is the ultrafilter generated by
the sets $\{\varphi(a):a\in\mathcal{F}|Y\}$, $Y\in\mathcal{U}|U$,
and $\mathcal{U}\re\mathcal{B}$
is  the ultrafilter on base set $\mathcal{B}$ generated by the sets $\mathcal{B}|Y:=\{b\in\mathcal{B}:b\sse Y\}$, where $Y\in\mathcal{U}$.
This is  a form of being a Fubini power of $\mathcal{U}$, with rank equal to the rank of the barrier $\mathcal{B}$.
\end{enumerate}

This proof outline turns out to work for many other cases of ultrafilters  associated to topological Ramsey spaces, as was first discovered in \cite{Dobrinen/Todorcevic14}.
The structures become more complex as we move towards ultrafilters with weaker partition relations, but for the cases of p-points investigated so far, we can find the exact structure of the RK classes inside the Tukey types.

%%%%%%%%%%%%%%%%%%%%%%%%
%%%%%%%%%%%%%%%%%%%%%%%%
%%%%%%%%%%%%%%%%%%%%%%%%
%%%%%%%%%%%%%%%%%%%%%%%%

\section{Forcing with Topological Ramsey Spaces}\label{sec.tRsoutline}

In the previous section we outlined some key properties of  forcing with the Ellentuck space partially ordered by almost inclusion, which is forcing equivalent to the partial ordering $\mathcal{P}(\om)/\Fin$.
We  saw
that  the Ramsey property is a sufficient combinatorial property to completely characterize when an ultrafilter is forced by $\mathcal{P}(\om)/\Fin$
over certain special inner models of the forms HOD$(\mathbb{R})^{V[G]}$ or $L(\mathbb{R})$.
 We outlined how theorems regarding canonical equivalence relations on barriers can be applied using  continuous cofinal maps to classify the ultrafilters which are Tukey reducible to a Ramsey ultrafilter.
Moreover, the methods employed
made clear
  the exact structure of the Rudin-Keisler classes inside the Tukey type of a Ramsey ultrafilter:
These are the isomorphism classes of the countable Fubini powers of the Ramsey ultrafilter.
These results turn out to be special cases of
more
 general phenomena arising when one forces ultrafilters using topological Ramsey spaces partially ordered by the $\sigma$-closed almost reduction.
In this section we provide an overview of these phenomena for abstract topological Ramsey spaces.

We first begin with the abstract definition of a topological Ramsey space from Todorcevic's book \cite{TodorcevicBK10}.

\subsection{Basics of  general topological Ramsey spaces}\label{sec.reviewtRs}

Building on earlier work of Carlson and Simpson in \cite{Carlson/Simpson90}, Todorcevic distilled  key properties of the Ellentuck space into  four axioms, \bf A.1  \rm -  \bf A.4\rm, which guarantee that a space is a topological Ramsey space.
The  axioms
are defined for triples
$(\mathcal{R},\le,r)$
of objects with the following properties.
$\mathcal{R}$ is a nonempty set,
$\le$ is a quasi-ordering on $\mathcal{R}$,
 and $r:\mathcal{R}\times\om\ra\mathcal{AR}$ is a mapping giving us the sequence $(r_n(\cdot)=r(\cdot,n))$ of approximation mappings, where
$\mathcal{AR}$ is  the collection of all finite approximations to members of $\mathcal{R}$.
For $a\in\mathcal{AR}$ and $A,B\in\mathcal{R}$,
\begin{equation}
[a,B]=\{A\in\mathcal{R}:A\le B\mathrm{\ and\ }(\exists n)\ r_n(A)=a\}.
\end{equation}

For $a\in\mathcal{AR}$, let $|a|$ denote the length of the sequence $a$.  Thus, $|a|$ equals the integer $k$ for which $a=r_k(a)$.
For $a,b\in\mathcal{AR}$, $a\sqsubseteq b$ if and only if $a=r_m(b)$ for some $m\le |b|$.
$a\sqsubset b$ if and only if $a=r_m(b)$ for some $m<|b|$.
For each $n<\om$, $\mathcal{AR}_n=\{r_n(A):A\in\mathcal{R}\}$.
\vskip.1in

\begin{enumerate}
\item[\bf A.1]\rm
\begin{enumerate}
\item
$r_0(A)=\emptyset$ for all $A\in\mathcal{R}$.\vskip.05in
\item
$A\ne B$ implies $r_n(A)\ne r_n(B)$ for some $n$.\vskip.05in
\item
$r_n(A)=r_m(B)$ implies $n=m$ and $r_k(A)=r_k(B)$ for all $k<n$.\vskip.1in
\end{enumerate}
\item[\bf A.2]\rm
There is a quasi-ordering $\le_{\mathrm{fin}}$ on $\mathcal{AR}$ such that\vskip.05in
\begin{enumerate}
\item
$\{a\in\mathcal{AR}:a\le_{\mathrm{fin}} b\}$ is finite for all $b\in\mathcal{AR}$,\vskip.05in
\item
$A\le B$ iff $(\forall n)(\exists m)\ r_n(A)\le_{\mathrm{fin}} r_m(B)$,\vskip.05in
\item
$\forall a,b,c\in\mathcal{AR}[a\sqsubset b\wedge b\le_{\mathrm{fin}} c\ra\exists d\sqsubset c\ a\le_{\mathrm{fin}} d]$.\vskip.1in
\end{enumerate}
\end{enumerate}

The number $\depth_B(a)$ is the least $n$, if it exists, such that $a\le_{\mathrm{fin}}r_n(B)$.
If such an $n$ does not exist, then we write $\depth_B(a)=\infty$.
If $\depth_B(a)=n<\infty$, then $[\depth_B(a),B]$ denotes $[r_n(B),B]$.

\begin{enumerate}
\item[\bf A.3] \rm
\begin{enumerate}
\item
If $\depth_B(a)<\infty$ then $[a,A]\ne\emptyset$ for all $A\in[\depth_B(a),B]$.\vskip.05in
\item
$A\le B$ and $[a,A]\ne\emptyset$ imply that there is $A'\in[\depth_B(a),B]$ such that $\emptyset\ne[a,A']\sse[a,A]$.\vskip.1in
\end{enumerate}
\end{enumerate}

If $n>|a|$, then  $r_n[a,A]$ denotes the collection of all $b\in\mathcal{AR}_n$ such that $a\sqsubset b$ and $b\le_{\mathrm{fin}} A$.

\begin{enumerate}
\item[\bf A.4]\rm
If $\depth_B(a)<\infty$ and if $\mathcal{O}\sse\mathcal{AR}_{|a|+1}$,
then there is $A\in[\depth_B(a),B]$ such that
$r_{|a|+1}[a,A]\sse\mathcal{O}$ or $r_{|a|+1}[a,A]\sse\mathcal{O}^c$.\vskip.1in
\end{enumerate}

The  {\em Ellentuck topology} on $\mathcal{R}$ is the topology generated by the basic open sets
$[a,B]$;
it extends the usual metrizable topology on $\mathcal{R}$ when we consider $\mathcal{R}$ as a subspace of the Tychonoff cube $\mathcal{AR}^{\bN}$.
Given the Ellentuck topology on $\mathcal{R}$,
the notions of nowhere dense, and hence of meager are defined in the natural way.
We  say that a subset $\mathcal{X}$ of $\mathcal{R}$ has the {\em property of Baire} iff $\mathcal{X}=\mathcal{O}\cap\mathcal{M}$ for some Ellentuck open set $\mathcal{O}\sse\mathcal{R}$ and Ellentuck meager set $\mathcal{M}\sse\mathcal{R}$.
A subset $\mathcal{X}$ of $\mathcal{R}$ is {\em Ramsey} if for every $\emptyset\ne[a,A]$,
there is a $B\in[a,A]$ such that $[a,B]\sse\mathcal{X}$ or $[a,B]\cap\mathcal{X}=\emptyset$.
$\mathcal{X}\sse\mathcal{R}$ is {\em Ramsey null} if for every $\emptyset\ne [a,A]$, there is a $B\in[a,A]$ such that $[a,B]\cap\mathcal{X}=\emptyset$.

\begin{defn}[\cite{TodorcevicBK10}]\label{defn.5.2}
A triple $(\mathcal{R},\le,r)$ is a {\em topological Ramsey space} if every subset of $\mathcal{R}$  with the property of Baire  is Ramsey and if every meager subset of $\mathcal{R}$ is Ramsey null.
\end{defn}

The following result can be found as Theorem
5.4 in \cite{TodorcevicBK10}.

\begin{thm}[Abstract Ellentuck Theorem]\label{thm.AET}\rm \it
If $(\mathcal{R},\le,r)$ is closed (as a subspace of $\mathcal{AR}^{\bN}$) and satisfies axioms {\bf A.1}, {\bf A.2}, {\bf A.3}, and {\bf A.4},
then every  subset of $\mathcal{R}$ with the property of Baire is Ramsey,
and every meager subset is Ramsey null;
in other words,
the triple $(\mathcal{R},\le,r)$ forms a topological Ramsey space.
\end{thm}

\begin{example}[The Ellentuck Space]\label{ex.Ell}
Putting the Ellentuck space into the notation just defined,
$\mathcal{R}=[\om]^{\om}$ and the partial ordering  $\le $ is  simply $\sse$.
Given $X\in[\om]^{\om}$,
listing the members of $X$ in increasing order as $\{x_i:i<\om\}$,
the $n$-th approximation to $X$ is $r_n(X)=\{x_i:i<n\}$.
In particular, $r_0(X)=\emptyset$.
$\mathcal{AR}_n=[\om]^n$ and $\mathcal{AR}=[\om]^{<\om}$.
For $a,b\in[\om]^{<\om}$, one may define
$a\le_{\mathrm{fin}} b$ if and only if $a\sse b$.
With  these definitions, $([\om]^{<\om},\sse,r)$ can be seen to satisfy \bf Axioms A.1 \rm - \bf A.4\rm, and it is recommended that the reader check this to build a basis for understanding of how the axioms work.
\end{example}

For the Ellentuck space, the definition of $\le_{\mathrm{fin}}$ is a bit flexible:
An alternate definition,  $a\le_{\mathrm{fin}}b$ if and only if $a\sse b$ and $\max(a)=\max(b)$, also
satisfies  the  \bf Axioms A.1 \rm - \bf A.4\rm.
In some spaces, the definition of $\le_{\mathrm{fin}}$ can be
more particular, for instance Milliken's space of strong trees.
However, for all the spaces in this article, except for Subsection \ref{subsec.finfin},
the definition of $\le_{\mathrm{fin}}$ is also flexible in a similar manner as for the Ellentuck space.

%%%%%%%%%%%%%%%%%%%%%

\subsection{Almost reduction, forced ultrafilters, and complete combinatorics}\label{subsec.almostred}

Let $(\mathcal{R},\le,r)$ be any topological Ramsey space.
The ordering $\le$ can be weakened to form a $\sigma$-closed order as follows.

\begin{defn}[Mijares, \cite{Mijares07}]\label{def.almostred}
For $X,Y\in\mathcal{R}$,
write $Y\le^* X$  if there is an $a\in\mathcal{AR}|Y$ such that $[a,Y]\sse [a,X]$.
In this case, we say that $Y$ is an {\em almost reduction} of $X$.
\end{defn}

In the Ellentuck space, $\le^*$ is simply $\sse^*$
where $A\sse^* B$ if and only if $A\setminus B$ is finite.
Mijares proved in \cite{Mijares07} that $(\mathcal{R},\le^*)$ is a $\sigma$-closed partial ordering.
Forcing with $(\mathcal{R},\le^*)$ yields a generic filter $\mathcal{U}$ on $\mathcal{R}$.
This filter induces an ultrafilter on the countable base set of first-approximations,
 $\mathcal{AR}_1:=\{r_1(X):X\in\mathcal{R}\}$,
as follows.

\begin{defn}[The ultrafilter on first approximations]\label{defn.uar1}
Let $\mathcal{U}$ be a generic filter for the forcing $(\mathcal{R},\le^*)$,
where $(\mathcal{R},\le,r)$ is some topological Ramsey space. $\mathcal{U}\re\mathcal{AR}_1$ denotes  the filter on base set $\mathcal{AR}_1$ generated by the sets
\begin{equation}
\mathcal{AR}_1|X:=\{r_1(Y):Y\le X\}, \ \  X\in\mathcal{U}.
\end{equation}
\end{defn}

\begin{fact}[Mijares, \cite{Mijares07}]\label{thm.uar1uf}
For $\mathcal{U}$  generic for $(\mathcal{R},\le^*)$,
$\mathcal{U}\re\mathcal{AR}_1$ is an ultrafilter.
\end{fact}

Since $\mathcal{U}$ is generic and since any subset of $\mathcal{F}\sse \mathcal{AR}_1$
induces the clopen (and hence property of Baire) set $\{Y\in\mathcal{R}:\exists a\in\mathcal{F}\, (a\sqsubset Y)\}$,
the Abstract Ellentuck Theorem along with
a density argument shows that
there is a member $X\in\mathcal{U}$ such that either $\mathcal{AR}_1|X\sse\mathcal{F}$ or else $\mathcal{AR}_1|X\cap\mathcal{F}=\emptyset$.
This shows that $\mathcal{U}\re \mathcal{AR}_1$ is an ultrafilter.

Such generic $\mathcal{U}$   and their induced ultrafilters $\mathcal{U}\re\mathcal{AR}_1$ satisfy strong properties.
First, they are  essentially p-points, but with respect to the space with its finite approximation structure from which they arise rather than with respect to $\om$.
By this, we mean that
for any  sequence $\lgl X_i:i<\om\rgl$ in $\mathcal{U}$ such that $X_i\ge^* X_{i+1}$ for all $i$, there is an $X\in \mathcal{U}$ such that $X\le^* X_i$ for each $i<\om$.
Such an $X$ is  pseudo-intersection of the sequence.
This often, though not always, will imply that $\mathcal{U}\re\mathcal{AR}_1$ is a p-point in the traditional sense, meaning that for each
 countable set of members $U_i\in\mathcal{U}\re\mathcal{AR}_1$,
there is a member $X\in\mathcal{U}\re\mathcal{AR}_1$ such that
 for each $i<\om$,
 all but finitely many members of  $X$ are  contained in  $U_i$.

The second and stronger property satisfied by ultrafilters forced by $(\mathcal{R},\le^*)$ is the analogue of selectivity, with respect to the topological Ramsey space.
In \cite{DiPrisco/Mijares/Nieto15},
the general definition of semiselective coideal is presented.
Here we shall only concentrate on when  the  generic filter $\mathcal{U}$ is selective.
The following definition is the natural abstraction of the formulation of selectivity due to  Mathias in \cite{Mathias77}.

\begin{defn}[\cite{DiPrisco/Mijares/Nieto15}]
\label{defn.seimselective}
Let $\mathcal{U}\sse\mathcal{R}$ be a generic filter  forced by $(\mathcal{R},\le^*)$.
Given a family $\mathcal{A}=\{A_a\}_{a\in\mathcal{AR}|A}\sse\mathcal{R}$,
we say that $Y\in\mathcal{R}$ is a {\em diagonalization} of $\mathcal{A}$ if for each $a\in\mathcal{AR}|Y$ we have $[a,Y]\sse[a,A_a]$.

A set $\mathcal{D}\sse$ is said to be {\em dense open} in $\mathcal{U}\cap [\depth_X(a),X]$
if
\begin{enumerate}
\item
(Density) For all $A\in  \mathcal{U}\cap [\depth_X(a),X]$, there is a $B\in\mathcal{D}$ such that $B\le A$;
\item
(Open)
For all $A\in  \mathcal{U}\cap [\depth_X(a),X]$,
for all  $B\in\mathcal{D}$, $A\le B$ implies $A\in\mathcal{D}$.
\end{enumerate}

Given $X\in\mathcal{U}$ and a collection $\vec{\mathcal{D}}=
\{\mathcal{D}_a\}_{a\in\mathcal{AR}|X}$ such that each $\mathcal{D}_a$ is open dense in
$\mathcal{U}\cap[\depth_X(a),X]$,
we say that $Y\in\mathcal{R}$
is a {\em diagonalization} of $\vec{\mathcal{D}}$ if there is a family $\mathcal{A}=\{A_a\}_{a\in\mathcal{AR}|X}$,
with $A_a\in\mathcal{D}_a$, such that
$Y$ is a diagonalization of $\mathcal{A}$.
\end{defn}

\begin{defn}[Abstract Selectivity, \cite{DiPrisco/Mijares/Nieto15}]\label{defn.selective}
A maximal filter $\mathcal{G}\sse\mathcal{R}$ is {\em selective} if
for each $A\in\mathcal{G}$ and each collection $\vec{\mathcal{D}}=
\{\mathcal{D}_a\}_{a\in\mathcal{AR}|A}$ such that each $\mathcal{D}_a$ is dense open in $\mathcal{G}\cap [\depth_A(a),A]$ and each $B\in\mathcal{G}|A$,
there is a $C\in\mathcal{G}$ such that $C$ is a diagonalization of $\mathcal{D}$ and $C\le B$.
\end{defn}

The following may be shown by  standard arguments, using genericity.

\begin{fact}
Each filter $\mathcal{U}\sse\mathcal{R}$ generic for $(\mathcal{R},\le^*)$ is selective.
\end{fact}

We now present the
version of complete combinatorics for topological Ramsey spaces.
This  result of Di Prisco, Mijares and Nieto  is in fact  more general,  pertaining to all semiselective coideals.
Here we only mention the case for  generic filters.

\begin{thm}[Di Prisco/Mijares/Nieto, \cite{DiPrisco/Mijares/Nieto15}]
\label{thm.completecomb}
If there exists a supercompact cardinal and $\mathcal{G}\sse\mathcal{R}$ is generic for $(\mathcal{R},\le^*)$,
then all definable subsets of $\mathcal{R}$ are $\mathcal{G}$-Ramsey.
Hence, each filter in $V$ on base set $\mathcal{R}$ which is selective is  generic for the forcing $(\mathcal{R},\le^*)$ over the Solovay model  $L(\mathbb{R})$.
\end{thm}

For certain topological Ramsey spaces, the Ellentuck space in particular, it is known that for a maximal filter $\mathcal{G}$ on the space, the strong form of selectivity above is equivalent to being Ramsey, meaning that for each $n<\om$, for any partition of $\mathcal{AR}_n$ into two sets  there is a member in $\mathcal{G}$ homogenizing the partition.
However, it is still open whether these two notions are the same for any topological Ramsey space.
Some equivalents of  Ramsey are proved in Chapter 2 of \cite{TrujilloThesis} and in Section 7 of \cite{Dobrinen/Mijares/Trujillo14} for certain collections of topological Ramsey spaces.

We close this subsection by pointing out one more  important similarity of general topological Ramsey spaces with the Ellentuck space.

\begin{thm}[Di Prisco/Mijares/Nieto, \cite{DiPrisco/Mijares/Nieto15}]
\label{thm.DMNcool}
Let $(\mathcal{R},\le,r)$ be a topological Ramsey space and
$\mathcal{U}\sse\mathcal{R}$ be  a selective ultrafilter in a transitive model  $M$ of ZF + DCR.
Let
$\mathbb{M}_{\mathcal{U}}$
be the Mathias-like forcing  consisting of basic open sets $[a,X]$, where $X\in\mathcal{U}$.
Then  forcing over $M$ with $\mathbb{M}_{\mathcal{U}}$ adds a generic $g\in\mathcal{R}$ with the property that $g\le^* A$ for each $A\in\mathcal{U}$.
Moreover,
$B\in\mathcal{R}$ is $\mathbb{M}_{\mathcal{U}}$-generic over $M$ if and only if $B\le^* A$ for  all $A\in\mathcal{U}$.
Furthermore, $M[\mathcal{U}][g]=M[g]$.
\end{thm}
For many further forcing properties  similar  Mathias forcing,
see  \cite{DiPrisco/Mijares/Nieto15}.

%%%%%%%%%%%%%%%%%%%%

\subsection{Continuous cofinal maps}

We saw in Theorem \ref{thm.20} that p-points have continuous Tukey reductions, by which we mean that
any monotone cofinal map from a p-point to another ultrafilter is  continuous when restricted below some member of the p-point.
This is actually a special case of a more general phenomenon, which is seen in the topological Ramsey spaces in \cite{Dobrinen/Todorcevic14}, \cite{Dobrinen/Todorcevic15},  \cite{TrujilloThesis}, \cite{Dobrinen/Mijares/Trujillo14}, and \cite{DobrinenJSL15}.
Recall that the metric topology on  $\mathcal{R}$ is the one induced by the basic open cones  $\{X\in\mathcal{R}:a\sqsubset X\}$, for $a\in\mathcal{AR}$.
This is the same as  topology inherited by $\mathcal{R}$ viewed as a subspace
of  $\mathcal{AR}^{\mathbb{N}}$ with
the product topology.

The terminology  {\em basic} was introduced by Solecki and Todorcevic in \cite{Solecki/Todorcevic04} in a study of Tukey structures of analytic posets.
When treating an ultrafilter as poset partially ordered by reverse inclusion,
it turns out that p-points are exactly ultrafilters which are basic, as was proved in \cite{Dobrinen/Todorcevic11}.
Since {\em p-point} has a very well-established meaning, in order to avoid any confusion,
 we will  use the terminology {\em closed under almost reduction} to refer to a filter $\mathcal{C}$ on $\mathcal{R}$ such that for each  sequence $\lgl X_i:i<\om\rgl$ of members of $\mathcal{C}$  such that $X_i\ge^* X_{i+1}$ for each $i<\om$,
there is a diagonalization $Y\in\mathcal{C}$ such that for each $i<\om$, $Y\le^* X_i$.
This property was called `selective' in \cite{Mijares07} and for the next five years following that paper.
Work of Trujillo in \cite{Trujillo16}
served to distinguish the two notions for a large class of topological Ramsey spaces, and now when we use the word {\em selective} we  are always referring to Definition \ref{defn.selective}.

\begin{defn}[\cite{Dobrinen/Mijares/Trujillo14}]\label{defn.basic}
 Assume that $\mathcal{C}\subseteq \mathcal{R}$ is a filter on $(\mathcal{R},\le)$. \emph{$\mathcal{C}$ has basic Tukey reductions} if whenever $\mathcal{V}$ is a non-principal ultrafilter on $\omega$ and  $f:\mathcal{C} \rightarrow \mathcal{V}$ is a monotone cofinal map, there is an $X\in \mathcal{C}$, a
 monotone map $f^*:\mathcal{C} \rightarrow \mathcal{V}$
which is continuous with respect to the metric topology on $\mathcal{R}$,
 and a function $\check{f}: \mathcal{AR} \rightarrow [\omega]^{<\omega}$  such that
\begin{enumerate}
\item  $f \upharpoonright ( \mathcal{C} \upharpoonright X)$ is continuous with respect to the metric topology on $\mathcal{R}$.
\item $f^*$ extends $f \upharpoonright ( \mathcal{C} \upharpoonright X)$ to $\mathcal{C}$.
\item
\begin{enumerate}
\item
$s\sqsubseteq t \in \mathcal{AR}$ implies that $\check{f}(s) \sqsubseteq \check{f}(t);$
\item For each $Y\in \mathcal{C}$, $f^*( Y) = \bigcup_{k<\omega} \check{f}(r_{k}(Y))$; and
\item
$\check{f}$ is monotone: If $s, t \in \mathcal{AR}$ with $s\le_{\fin} t$, then $\check{f}(s) \subseteq \check{f}(t)$.
\end{enumerate}
\end{enumerate}
\end{defn}

The following  general theorem
 encompasses all known examples and 
provides a weak condition under which a maximal  filter $\mathcal{U}$  on $\mathcal{R}$ has basic Tukey reductions.
In the following, a member $A$ of  $\mathcal{U}$ is fixed, and for each $X\in\mathcal{U}$ such that $X\le A$,
$d(X)$ denotes the set $\{\depth_{A}(r_n(X)):n<\om\}$, the collection of the depths of the finite approximations to $X$ with respect to the fixed  $A$.
In the case of the Ellentuck space, each $d(X)$ is simply $X$.
It  very well may turn  out to be the case that  the requirement that  $\{d(X):X\in\mathcal{G}\}$ generates a nonprincipal ultrafilter on $\om$
is
  simply  true for all topological Ramsey spaces.
The following theorem   extends analogous theorems in \cite{Dobrinen/Todorcevic14} and \cite{Dobrinen/Todorcevic15},
and is attributed to Trujillo in \cite{Dobrinen/Mijares/Trujillo14},  where 
  a more general statement of the following  can be found.

\begin{thm}[Trujillo, \cite{Dobrinen/Mijares/Trujillo14}]\label{thm.tim}
Let  $(\mathcal{R},\le,r)$ be any topological Ramsey space.
If $\mathcal{U}$ is generic for the forcing $(\mathcal{R},\le^*)$
and $\{d(X):X\in\mathcal{U}\}$ generates a nonprincipal ultrafilter on $\om$,
then $\mathcal{U}$ has basic Tukey reductions.
\end{thm}

\subsection{Barriers and canonical equivalence relations}

The notions of fronts and barriers on the Ellentuck space can be extended to  topological Ramsey spaces.
The Abstract Nash-Williams Theorem (Theorem 5.17  in \cite{TodorcevicBK10}) follows from the Abstract Ellentuck Theorem, just as the  Nash-Williams Theorem follows from  Ellentuck's Theorem.
For many arguments involving the Tukey  and Rudin-Keisler structures of forced ultrafilters, the full strength of the Abstract Ellentuck Theorem is not used, but rather  Abstract Nash-Williams Theorem suffices.

\begin{defn}[\cite{TodorcevicBK10}]\label{defn.5.16}
A family $\mathcal{F}\sse\mathcal{AR}$ of finite approximations is
\begin{enumerate}
\item
{\em Nash-Williams} if $a\not\sqsubseteq b$ for all $a\ne b$ in $\mathcal{F}$;
\item
{\em Sperner} if $a\not\le_{\mathrm{fin}} b$ for all $a\ne b$ in $\mathcal{F}$;
\item
{\em Ramsey} if for every partition $\mathcal{F}=\mathcal{F}_0\cup\mathcal{F}_1$ and every $X\in\mathcal{R}$,
there are $Y\le X$ and $i\in\{0,1\}$ such that $\mathcal{F}_i|Y=\emptyset$.
\end{enumerate}
\end{defn}

\begin{defn}\label{def.frontR1}
Suppose $(\mathcal{R},\le,r)$ is a closed triple that satisfies {\bf A.1} - {\bf A.4}.
Let $X\in\mathcal{R}$.
A family $\mathcal{F}\sse\mathcal{AR}$ is a {\em front} on $[0,X]$ if
\begin{enumerate}
\item
For each $Y\in[0,X]$, there is an $a\in \mathcal{F}$ such that $a\sqsubset Y$; and
\item
$\mathcal{F}$ is Nash-Williams.
\end{enumerate}
A family $\mathcal{B}\sse\mathcal{AR}$ is a {\em barrier} on $[0,X]$ if it satisfies (1) and also satisfies
\begin{enumerate}
\item[($2'$)]
$\mathcal{B}$ is {\em Sperner}.
\end{enumerate}
\end{defn}

\begin{thm}[Abstract Nash-Williams Theorem]\label{thm.ANW}
Suppose $(\mathcal{R},\le,r)$ is a closed triple that satisfies {\bf A.1} - {\bf A.4}. Then
every Nash-Williams family of finite approximations is Ramsey.
\end{thm}

It is also proved in \cite{TodorcevicBK10}
that whenever the quasi-ordering $\le_{\mathrm{fin}}$ is a partial ordering,
then an abstract version of Galvin's Lemma holds:
For each front $\mathcal{F}$ on $\mathcal{R}$, there is a member $Y\in\mathcal{R}$ such that $\mathcal{F}|Y:=\{a\in \mathcal{F}: a\in\mathcal{AR}|Y\}$
is a barrier on $[0,Y]$.
In all spaces that we have worked on so far, the quasi-order $\le_{\mathrm{fin}}$ actually is a partial ordering.
Thus, we are free to interchange usages of fronts with barriers as they do not affect any results.
Analogously to how $[\om]^k$ is the uniform barrier of rank $k$ on the Ellentuck space,
for any topological Ramsey space,
 $\mathcal{AR}_k$ is the uniform front of rank $k$.
One can define by recursion  fronts of all countable ranks on abstract topological Ramsey spaces, though the definition requires more care in the abstract setting.

 The theorem of \Pudlak\ and \Rodl\ in Subsection \ref{subsection.canonicaleqrel}
 canonizing equivalence relations on fronts  on $[\om]^{\om}$  generalizes to a large class of topological Ramsey spaces.
This has been seen in
\cite{Lefmann96} and more generally in
\cite{Klein/Spinas05}
for the Milliken space of infinite block sequences
and more recently in  \cite{Dobrinen/Todorcevic14}, \cite{Dobrinen/Todorcevic15}, \cite{Dobrinen/Mijares/Trujillo14}, \cite{DobrinenJSL15} and \cite{DobrinenJML16},
where new topological Ramsey spaces were constructed which are dense inside certain $\sigma$-closed forcings producing ultrafilters satisfying weak partition relations.
A sampling of the
 exact formulations
of these canonical equivalence relations will be explicated in the following sections.
The rest
are left for the reader to find in the original sources, as reproducing them here would require too much space.
The point we want to make is that such theorems for certain spaces and, in all these cases, have the general form that the canonical equivalence relations are essentially
given by projections to substructures.
If the Ramsey space has members which have  a tree-like structure, then the canonical equivalence relations  are defined by projections to subtrees.
If the members of the Ramsey space are sequences of ordered structures, then the canonical equivalence relations are defined by  projections to  substructures.
This behavior is what allows for the Rudin-Keisler structure inside the Tukey types to be deduced.

\subsection{Initial Tukey and Rudin-Keisler structures, and Rudin-Keisler structures inside Tukey types}

Recall from Subsection \ref{subsection.initialTRKstructure}
  that if $\mathcal{U}$ is a Ramsey ultrafilter and $\mathcal{V}\le_T\mathcal{U}$,
then
there is a front $\mathcal{F}$ and a function $g:\mathcal{F}\ra \om$ such that
$\mathcal{V}$ is actually equal to the
ultrafilter $g(\lgl\mathcal{U}\re\mathcal{F}\rgl)$, which is the filter on base set $\{g(a):a\in\mathcal{F}\}$ generated by the sets
$\{g(a):a\in\mathcal{F}|X\}$, $X\in\mathcal{U}$.
From this, one obtains that if $\mathcal{V}$ is nonprincipal, then  $\mathcal{V}\equiv_T\mathcal{U}$
and $\mathcal{V}\equiv_{RK}\mathcal{U}^{\al}$ for some $\al<\om_1$.
This is  a particular instance of a more general phenomenon which has been used successfully in several classes of ultrafilters to classify the structure of those ultrafilters Tukey reducible to an ultrafilter forced by a topological Ramsey space.

The steps (1) - (5) outlined in Subsection \ref{subsection.initialTRKstructure}
for Ramsey ultrafilters on $\om$ also provide the outline for obtaining precise results about Tukey and Rudin-Keisler structures below ultrafilters forced by topological Ramsey spaces.
A collection of ultrafilters
closed under Tukey reduction  is called an {\em initial Tukey structure}.
Likewise, a collection of ultrafilters which is closed under Rudin-Keisler reduction is called an {\em initial RK structure}.
Given a topological Ramsey space $\mathcal{R}$ and a filter $\mathcal{U}$ generic for $(\mathcal{R},\le^*)$,
we are interested in
classifying the initial Tukey an initial RK structures below the ultrafilter $\mathcal{U}\re\mathcal{AR}_1$.
Since $\mathcal{U}$ and $\mathcal{U}\re\mathcal{AR}_1$ are Tukey equivalent, we may work with either in the analysis.

The following is a special case of a more general fact  shown in \cite{Dobrinen/Mijares/Trujillo14},  stated here only for the case of generic filters.
Its proof used Theorem \ref{thm.tim}
showing that in particular, generic filters have continuous cofinal Tukey reductions.
A {\em front on $\mathcal{U}$} is a family $\mathcal{F}\sse\mathcal{AR}$ such that
for each $X\in\mathcal{U}$, there is an $a\in\mathcal{F}$ such that $a\sqsubset X$,
and no member of $\mathcal{F}$  is a proper initial segment of another member of $\mathcal{F}$.
Recall that $g(\mathcal{F}|X)=\{g(a):a\in\mathcal{F}|X\}$.

\begin{thm}[\cite{Dobrinen/Mijares/Trujillo14}]\label{thm.50}
Let $\mathcal{U}$ be generic for $(\mathcal{R},\le^*)$, and suppose that $\mathcal{V}$ is a nonprincipal ultrafilter on base set $\om$ such that $\mathcal{V}\le_T\mathcal{U}$.
Then there is a front $\mathcal{F}$ on $\mathcal{U}$ and a function $g:\mathcal{F}\ra \om$ such that for each $V\in\mathcal{V}$, there is an $X\in\mathcal{U}$ such that $g(\mathcal{F}|X)\sse V$, and moreover each such $g(\mathcal{F}|X)$ is a member of $\mathcal{V}$.
Thus, $\mathcal{V}$ equals the ultrafilter
 on the base set $\{g(a):a\in\mathcal{F}\}$
generated by the set $\{g(\mathcal{F}|X):X\in\mathcal{U}\}$.
\end{thm}

Since each function $g:\mathcal{F}\ra\om$ induces an equivalence relation on $\mathcal{F}$,
once a canonization theorem is proved for equivalence relations on  fronts,
it will help us understand
sets of the form $\{g(a):a\in\mathcal{F}\}$.
Once these are well-understood,
the  inital Tukey and initial RK structures are well-understood.
This will be made clear for a sampling of examples in the next Sections.

%%%%%%%%%%%%%%%%%%%%%%%%
%%%%%%%%%%%%%%%%%%%%%%%%
%%%%%%%%%%%%%%%%%%%%%%%%
%%%%%%%%%%%%%%%%%%%%%%%%

\section{Topological Ramsey space theory applied to ultrafilters satisfying weak partition relations: an overview of the following sections}\label{sec.forcingTRS}

In Section \ref{sec.RamseyUF}  we delineated some of the important properties of Ramsey ultrafilters and how these properties are connected with the Ellentuck space.
We saw that Ramsey ultrafilters have complete combinatorics,  are RK minimal, Tukey minimal, and that the RK classes inside a Ramsey ultrafilter's  Tukey type  are exactly those of its countable Fubini powers.
We also saw how the Nash-Williams Theorem provides a quick proof that any ultrafilter forced by $\mathcal{P}(\om)/\Fin$ is Ramsey.
In the next two sections, we will look in-depth at some examples of ultrafilters from the literature which are p-points satisfying  weak partition relations.
For these examples, we will go through the steps of the results that these ultrafilters are similar to Ramsey ultrafilters, achieving Objectives 1\,-\,4 below.
In the final section of this expository paper, we will mention similar results for some broad classes of ultrafilters.
Some of these ultrafilters are well-known from the literature and some are new, having been produced by the construction of new topological Ramsey spaces.
These results  were motivated by the following question.

\begin{question}\label{question.motivator}
 What  is the structure of the Tukey types of ultrafilters which are close to minimal in the Rudin-Keisler hierarchy?
\end{question}

This  is closely related to the question of finding the Tukey structure of ultrafilters satisfying weak partition relations, since in all known examples,  partition relations are corollated with  increased  Rudin-Keisler strength.
In investigating these questions, the following objectives were attained for the ultrafilters investigated so far.

\begin{Objective1}
Complete combinatorics.
\end{Objective1}

\begin{Objective2}
The exact structure of the Tukey types  of all ultrafilters Tukey reducible to the forced ultrafilter.
\end{Objective2}

\begin{Objective3}
The exact structure of the Rudin-Keisler classes of all ultrafilters Rudin-Keisler reducible to the forced ultrafilter.
\end{Objective3}

\begin{Objective4}
The exact structure of the Rudin-Keisler classes inside the Tukey type of each ultrafilter Tukey reducible to the forced ultrafilter.
\end{Objective4}

To  attain these goals  for ultrafilters which are
 constructed using some $\sigma$-closed partial order by forcing or by using  the Continuum Hypothesis, Martin's Axiom, or some weaker cardinal invariant assumption to construct them, it suffices to
satisfy the steps (1) - (5)  outlined
in Subsection \ref{subsection.initialTRKstructure}  which were  generalized to abstract topological Ramsey spaces  in Section \ref{sec.tRsoutline}.
First, and sometimes challenging, one  needs to
 find a topological Ramsey space which forces the ultrafilter under investigation.
It suffices to show that there is a dense set in the partial order which  can be structured to form a topological Ramsey space.
This attains Objective 1.
Second, one needs to prove that cofinal maps from such ultrafilters are continuous, with respect to the correct topology.
This has now been done in much generality in
Theorem \ref{thm.tim}.
Third, one needs to
 prove
that equivalence relations on fronts for these spaces are canonical when restricted below some member of the space.
One needs to understand  these canonical equivalence relations very well in order to analyze their implications.
Putting this together with Theorem \ref{thm.50}
allows us to achieve Objectives 2 - 4.

%%%%%%%%%%%%%%%%%%%%%%%

\section{Weakly Ramsey ultrafilters}\label{section.R_1}

An ultrafilter $\mathcal{U}$ is {\em weakly Ramsey}
if for each $l\ge 2$ and each coloring $c:[\om]^2\ra l$, there is a member $U\in\mathcal{U}$ such that the restriction of $c$ to $[U]^2$ takes on at most two colors.
This is denoted symbolically as
\begin{equation}
\mathcal{U}\ra(\mathcal{U})^2_{l,2}.
\end{equation}
Blass showed in \cite{Blass74} that  weakly Ramsey ultrafilters have exactly one Rudin-Keisler predecessor, and that is a Ramsey ultrafilter.
Thus, the initial Rudin-Keisler structure below an ultrafilter forced by $\bP_1$ is simply a chain of length two (disregarding the principal ultrafilters).

In \cite{Laflamme89}, Laflamme constructed a partial ordering $(\bP_1,\le_1)$ which forces a weakly Ramsey ultrafilter.
This partial ordering  has conditions which are simply infinite subsets of $\om$, but the partial order  $\le_1$  is stronger than inclusion.

\begin{defn}[$(\bP_1,\le_1)$, \cite{Laflamme89}]\label{defn.le1}
$\bP_1=[\om]^{\om}$.
Let $X,Y\in[\om]^{\om}$.
Enumerate them
in increasing order  and in blocks of increasing size as $X=\lgl x_1^1,x_1^2,x_2^2,\dots,x^n_1,\dots,x^n_n,\dots\rgl$,
and $Y=\lgl y_1^1,y_1^2,y_2^2,\dots,y^n_1,\dots,y^n_n,\dots\rgl$.
We call $\{x^n_1,\dots,x^n_n\}$ the $n$-th block of $X$ and similarly for $Y$.
Define $Y\le_1 X$ if and only if
\begin{equation}
\forall n\ \exists m \
\{y_1^n,\dots,y_n^n\}\sse\{x_1^m,\dots,x_m^m\}.
\end{equation}
\end{defn}
Note that $Y\le_1 X$ implies $Y\sse X$ but not vice versa.
It is this stronger ordering which is responsible for producing an ultrafilter which is weakly Ramsey but not Ramsey.

 Laflamme gave a combinatorial characterization of
ultrafilters forced by $\mathbb{P}_1$,  showing that these ultrafilters have complete combinatorics in the original sense of  Blass.
An ultrafilter $\mathcal{V}$ {\em satisfies  the Ramsey partition relation} RP$(k)$ if
for all functions $f:[\om]^k\ra 2$ and all partitions $\lgl A_m:m<\om\rgl$ of $\om$ with each $A_m\not\in\mathcal{V}$,
there is a set $X\in\mathcal{V}$ such that
\begin{enumerate}
\item
$|X\cap A_m|<\om$ for all $m<\om$;
\item
$|f''[X\cap A_m]^k|\le 1$ for all $m<\om$.
\end{enumerate}
Baumgartner and Taylor showed in \cite{Baumgartner/Taylor78}
that RP$(2)$ is equivalent to weakly Ramsey (see also \cite{Blass74}).
Laflamme also proved in \cite{Laflamme89} that for each $k$, RP$(k)$ is equivalent to $\mathcal{U}\ra(\mathcal{U})^k_{l,2^{k-1}}$, a fact which  he credits to Blass.

Laflamme  proved that  the ultrafilter forced by $\mathbb{P}_1$ satisfies RP$(k)$ for all $k$.
Thus, it would seem that the ultrafilters forced by
$\bP$ are a strong sort of  weakly Ramsey ultrafilter.
Indeed, it follows by work of Trujillo in  that, assuming CH, there are weakly Ramsey ultrafilters which  do not satisfy RP$(k)$ for some $k>2$.
(See Corollary 4.2.5 in \cite{TrujilloThesis} for the exact statement.)
The properties RP$(k)$ for all $k$ are what completely characterize ultrafilters being forced by $\bP_1$ over a canonical inner model, as follows.
Recall that an ultrafilter $\mathcal{U}$ is {\em rapid} if for each strictly increasing function  $h:\om\ra\om$,
there is a member $X\in\mathcal{U}$ such that for each $n<\om$,  $|X\cap h(n)|\le n$.

\begin{thm}[Laflamme, \cite{Laflamme89}]
Let $\kappa$ be a Mahlo cardinal and $G$ be generic for the L\'{e}vy collapse of $\kappa$.
If $\mathcal{U}\in V[G]$ is a rapid ultrafilter satisfying RP$(k)$ for all $k$, but $\mathcal{U}$ is not Ramsey, then $\mathcal{U}$ is $\bP_1$-generic over HOD$(\mathbb{R})^{V[G]}$.
\end{thm}

The proof uses a key theorem Laflamme proves earlier in \cite{Laflamme89} which is worth mentioning, as the reader will see the correlation with the  topological Ramsey space formulation  presented shortly.

\begin{thm}[Laflamme, \cite{Laflamme89}]
Let $\mathcal{U}$ be a nonprincipal ultrafilter.
Then the following are equivalent:
\begin{enumerate}
\item
$\mathcal{U}$ is rapid and satisfies RP$(k)$ for all $k$.
\item
For all $\Sigma_1^1$ sets $\mathcal{X}\sse[\om]^{\om}$,
there is a set $X\in\mathcal{U}$ such that
\begin{equation}
\{Y\in[\om]^{\om}:Y\le_1 X\}\sse\mathcal{X}\mathrm{\ \ or\ \ }
\{Y\in[\om]^{\om}:Y\le_1 X\}\cap\mathcal{X}=\emptyset.
\end{equation}
\item
$\mathcal{U}$ is rapid and satisfies $\mathcal{U}\ra(\mathcal{U})^k_{l,2^{k-1}}$ for all $k$.
\end{enumerate}
\end{thm}

It is condition (2) that we shall soon see  is very closely related to the topological Ramsey space formulation of complete combinatorics.

The following topological Ramsey space
was constructed to essentially form a dense subset of $\bP_1$, so that the Ramsey space is  forcing equivalent to $\bP_1$.

\begin{defn}[$(\mathcal{R}_1,\le_{\mathcal{R}_1},r)$, \cite{Dobrinen/Todorcevic14}]\label{defn.R_1}
Let
\begin{equation}
\mathbb{T}=\{\lgl\rgl\}\cup\{\lgl n\rgl:n<\om\}\cup\bigcup_{n<\om}\{\lgl n,i\rgl:i\le n\}.
\end{equation}
$\mathbb{T}$ is an infinite tree of height two and consists of  an
infinite sequence of finite trees,
where the {\em $n$-th subtree} of $\bT$ is
\begin{equation}
\bT(n)=\{\lgl\rgl,\lgl n\rgl\}\cup\{\lgl n,i\rgl:i\le n\}.
\end{equation}
Thus $\bT(n)$  is a finite tree of height two with one node  $\lgl \rgl$  on level $0$, one node $\lgl n\rgl$ on level  $1$,  and $n+1$ nodes on level $2$.
An infinite subtree $X\sse \bT$ is a member of $\mathcal{R}_1$ if and only if
$X$  is tree-isomorphic to $\bT$.
This means  $X$ must be an infinite sequence of finite subtrees such that the $n$-th subtree $X(n)$
has the node $\lgl\rgl$ on level $0$, one node $\lgl k_n\rgl$ on level $1$,
and $n+1$ many nodes $\{\lgl k_n, i\rgl:i\in I_n\}$
on level $2$,
where $I_n$ is some subset of  $k_n+1$ of size $n+1$.
Moreover, we require that  for $n<n'$,
$X(n)$ and $X(n')$ come from subtrees $\bT(k_n)$ and $\bT(k_{n'})$ where $k_n<k_{n'}$.

The partial ordering $\le_{\mathcal{R}_1}$ on $\mathcal{R}_1$ is simply that of subtree.
The restriction map $r$ is defined by
$r_n(X)=\bigcup_{i<n}X(i)$, for each $n<\om$ and $X\in\mathcal{R}_1$.
$\mathcal{AR}_n$ is the set  $\{r_n(X):X\in\mathcal{R}_1\}$;
 $\mathcal{AR}=\{r_n(X):X\in\mathcal{R}_1,\ n<\om\}$.
For $a,b\in\mathcal{AR}$,
$b\le_{\mathrm{fin}} a$ if and only if
$b$ is a subtree of $a$.
The basic open sets are given by
$[a,X]=\{Y\in\mathcal{R}_1:a\sqsubseteq Y$ and $Y\le_{\mathcal{R}_1} X\}$.
\end{defn}

$\bT$ is a tree which codes Laflamme's blocking structure.
Instead of taking all infinite sets and ordering them by a partial ordering stricter than inclusion,
the shape of the trees allowed in $\mathcal{R}_1$ transfers the strict partial ordering $\le_1$ of $\bP_1$  to the structure of the trees.
By restricting $\mathcal{R}_1$ to contain only those subtrees $X$ of $\bT$ which  have
each $n$-th subtree $X(n)$ coming from within one $m$-th subtree $\bT(m)$,
the partial ordering  $\le_1$ gets transfered to the structure of the tree.
The further restriction that each $n$-th subtree of $X$, $X(n)$, must come from a different subtree of $\bT$ further serves to simplify the set of trees we work with, and
more importantly,  aids in proving  that the Pigeonhole Principle, \bf Axiom A.4\rm, holds for this space.

For finite sequences $a\in\mathcal{AR}$,
we shall write $a=(a(0),\dots,a(k-1))$ to denote that $a$ is a sequence of length $k$ of subtrees where  the $i$-th subtree has $i+1$ many maximal nodes.
We now show how
 \bf Axiom A.4 \rm follows from the finite Ramsey Theorem.
Let
$X\in\mathcal{R}_1$,   $k<\om$,
$Y\le_{\mathcal{R}_1} X$, and
$a=r_k(Y)$.
Let $\mathcal{O}\sse\mathcal{AR}_{k+1}$ be given.
Let $m$ be the least integer such that $a\sse r_m(X)$.
In other words, $\depth_X(a)$ is  finite.
To show \bf A.4 \rm
we need to show  there is some $Z\in[r_m(X),X]$ such that  either
$r_{k+1}[a,Z]\sse\mathcal{O}$
or $r_{k+1}[a,Z]\cap\mathcal{O}=\emptyset$.
Notice that  set $r_{k+1}[a,X]$ is  the set of all $c\in\mathcal{AR}_{k+1}$ such that $(c(0),\dots, c(k-1))=(a(0),\dots,a(k-1))$.

Let $\mathcal{R}_1(k)$ denote the set of all $c(k)$ where $c\in\mathcal{AR}_{k+1}$, that is the set of all $k$-th trees of some member of $\mathcal{AR}_{k+1}$.
The set $\mathcal{O}$ induces a coloring on
\begin{equation}
\mathcal{R}_1(k)\re (X/m):=
\{c(k)\in \mathcal{R}_1(k):\exists i\ge m\, (c(k)\sse X(i))\},
\end{equation}
since
\begin{equation}
r_{k+1}[a,X]=\{(a(0),\dots,a(k-1),c(k)):c(k)\in
\mathcal{R}_1(k)\re (X/m)\}.
\end{equation}
Given $c(k)\in\mathcal{R}_1(k)\re (X/m)$,
define $f(c(k))=0$ if $(a(0),\dots,a(k-1),c(k))\in\mathcal{O}$,
and $f(c(k))=1$ if
$(a(0),\dots,a(k-1),c(k))\not\in\mathcal{O}$.
Thus, to construct a $Z\in[r_m(X),X]$ for which either $r_{k+1}[a,Z]\sse\mathcal{O}$ or $r_{k+1}[a,Z]\cap\mathcal{O}=\emptyset$,
it suffices to construct such a $Z$ with the property that $f$ is constant on the  $\mathcal{R}_1(k)\re (Z/m)$.

This follows from the finite Ramsey Theorem.
Let $r_m(Y)=r_m(X)$.
Given $r_j(Y)$, to construct $Y(j)$
take some $l$ large enough that coloring the $k+1$ sized subsets of a set of size $l$, there is a subset of size $j$ on which the $k+1$ sized subsets are homogeneous.
Take some subtree  $Y(j)$ of $Y(l)$ with $j$ many maximal nodes such that $f$ is homogeneous
on the set of $c(k)\in \mathcal{R}_1(k)$ such that $c(k)$ is a subtree of $Y(j)$.
This constructs $r_{j+1}(Y)$.
Let $Y$ be the infinite sequence of the $Y(i)$, $i<\om$.
This $Y$ is a member of $\mathcal{R}_1$.

Now the color of $f$ might be different for $c(k)$'s coming from within different subtrees of $Y$.
But there must be infinitely many $j$ for which the color of $f$ on
$\{c(k)\in\mathcal{R}_1(k):c(k)\sse Y(j)\}$ is the same.
Take an increasing sequence  $(j_i)_{i\ge m}$
 and thin each $Y(j_i)$ to  a subtree $Z(i)$ with $i+1$ many maximal nodes.
Let $Z=(X(0),\dots, X(m-1),Z(m),Z(m+1),\dots)$.
Then $Z$ is  a member of $[r_m(X),X]$ for which $f$ is constant on $\mathcal{R}_1(k)\re (Z/m)$.
Hence, $r_{k+1}[a,Z]$ is either contained in or disjoint from  $\mathcal{O}$.
This proves \bf A.4\rm.  The other three axioms are routine to prove, and it is suggested that the reader go through those proofs to build intuition.

\begin{thm}[Dobrinen/Todorcevic, \cite{Dobrinen/Todorcevic14}]
$(\mathcal{R}_1,\le_{\mathcal{R}_1},r)$ is a topological Ramsey space.
\end{thm}

The topological Ramsey space constructed has the property that below any member $S$ of $\bP_1$,
there is a correspondence between $\mathcal{R}_1$ and a dense set below $S$.

\begin{fact}
 $(\mathcal{R}_1,\sse)$ is forcing equivalent to $(\bP_1,\le_1)$.
\end{fact}

\begin{proof}
Let $S$ be any member of $[\om]^{\om}$.
Enumerate  $S$ in   blocks  of increasing length as $\{s^1_1,s^2_1,s^2_2,s^3_1,\dots\}$
as in Laflamme's blocking procedure.
Let $\theta:S\ra\bT$ be the map which takes $s^n_i$ to the node $\lgl n-1,i-1\rgl$ in $[\bT]$.
(Recall that $[\bT]$ denotes the set of maximal nodes in $\bT$.)
Then for each member $S'\in[\om]^{\om}$ such that $S'\le_1 S$,
$\theta(S')$ induces the  subtree  $\widehat{\theta(S')}$ of $\bT$ consisting of  the set of all initial segments of members of $\theta(S')$.
Now this $\widehat{\theta(S')}$ might not actually be a member of $\mathcal{R}_1$ as
there could be two  blocks of $S'$ that lie in one block of $S$;
this would translate to two subtrees
$\widehat{\theta(S')}(m)$ and $\widehat{\theta(S')}(n)$, for some $m<n$, being subtrees of the same $\bT(k)$ for some $k$.
However, we can take a subtree $X\sse \widehat{\theta(S')}$ which has each $n$-th subtree of $X$ coming from a different $k$-th subtree of $\bT$ so that $X\in\mathcal{R}_1$.
Then $S'':=\theta^{-1}(X)$
will be a member of $[\om]^{\om}$ such that $S''\le_1 S$, and
$\theta(S'')=[X]$, where $X$
is a member of $\mathcal{R}_1$.
Thus,  given an $S\in\bP_1$,
the set $\{\theta^{-1}([X]):X\in\mathcal{R}_1\}$ is dense below $S$ in the partial ordering $\bP_1$.
\end{proof}

For $\mathcal{R}_1$, the $\sigma$-closed almost reduction ordering
presented in Definition \ref{def.almostred}
is equivalent to the following.
Given $X,Y\in\mathcal{R}_1$,
$Y \le^*_{\mathcal{R}_1} X$ if and only if
there is some $m$ such that for all $n\ge m$, there is an $i_n$ such that $Y(n)\sse X(i_n)$.
Notice that for the space $\mathcal{R}_1$,
$Y \le^*_{\mathcal{R}_1} X$ if and only if
$[Y]\sse^*[X]$.

Let $\mathcal{U}_{\mathcal{R}_1}$ denote a generic ultrafilter forced by $(\mathcal{R}_1,\le_{\mathcal{R}_1}^*)$.
The following from \cite{Dobrinen/Todorcevic14} completely characterizes the Tukey types  of the ultrafilters Tukey reducible to $\mathcal{U}_{\mathcal{R}_1}$.
Furthermore, it characterizes the Rudin-Keisler classes inside those Tukey types.
The set $\mathcal{Y}_{\mathcal{R}_1}$  denotes a countable set of p-points which will be defined below.

\begin{thm}[Dobrinen/Todorcevic, \cite{Dobrinen/Todorcevic14}]\label{thm.DTTams1}
$(\mathcal{R}_1,\le_{\mathcal{R}_1})$   and $(\bP_1,\le_1)$ are forcing equivalent, and
 $(\mathcal{R}_1,\le_{\mathcal{R}_1}^*)$   and $(\bP_1,\le_1^*)$ are forcing equivalent.
Let $\mathcal{U}_{\mathcal{R}_1}$ denote the ultrafilter forced by
$(\mathcal{R}_1,\le_{\mathcal{R}_1}^*)$,
and
let $\mathcal{U}_0$ denote the Ramsey ultrafilter obtained from projecting  $\mathcal{U}_{\mathcal{R}_1}$ to level $1$ on the tree $\bT$.
\begin{enumerate}
\item
For each nonprincipal ultrafilter  $\mathcal{V}\le_T\mathcal{U}_{\mathcal{R}_1}$,
either $\mathcal{V}\equiv_T\mathcal{U}_{\mathcal{R}_1}$ or else $\mathcal{V}\equiv_T\mathcal{U}_0$.
\item
If $\mathcal{V}\equiv_T\mathcal{U}_{\mathcal{R}_1}$,
then $\mathcal{V}$ is Rudin-Keisler equivalent to some Fubini iterate of ultrafilters each of which is in
a specific countable collection of p-points, denoted $\mathcal{Y}_{\mathcal{R}_1}$.
If $\mathcal{V}\equiv_T\mathcal{U}_0$,
then $\mathcal{V}$ is Rudin-Keisler equivalent to some $\al$-th Fubini power of $\mathcal{U}_0$,
where $\al<\om_1$.
\end{enumerate}
\end{thm}

The proof follows the same five steps as in Subsection \ref{subsection.initialTRKstructure}.
For the rest of this section, let $\mathcal{U}$ denote $\mathcal{U}_{\mathcal{R}_1}$.
Let $\mathcal{V}$ be a nonprincipal ultrafilter on base set $\om$ such that  $\mathcal{V}\le_T\mathcal{U}$.
For Step (1),  it is shown that
for each monotone cofinal map $f:\mathcal{U}\ra\mathcal{V}$,
there is a member $X\in\mathcal{U}$ such
$f$ is continuous when restricted to $\mathcal{U}\re X$.
With a bit more work, one can show that in fact there is montone cofinal  map $f:\mathcal{R}_1\ra [\om]^{<\om}$ witnessing this Tukey reduction which is continuous with respect to the metric topology on $\mathcal{R}_1$;
that is the topology generated by basic open sets of the form $[a,\bT]$, where $a\in\mathcal{AR}$.
Thus, there is a finitary map $\hat{f}:\mathcal{AR}\ra[\om]^{<\om}$
which recovers $f$ as follows:
for each $X\in\mathcal{R}_1$,
\begin{equation}
f(X)=\bigcup_{k<\om}\hat{f}(r_k(X)).
\end{equation}

For Step (2),
let $\mathcal{F}$ be the front on $\mathcal{R}_1$ defined as follows:
For each $X\in\mathcal{R}_1$,
let $n(X)$ be the least integer such that $\hat{f}(r_{n(X)}(X))\ne\emptyset$.
Define
\begin{equation}
\mathcal{F}=\{r_{n(X)}(X):X\in\mathcal{R}_1\}.
\end{equation}
Then $\mathcal{F}$ is a front, since each member of $\mathcal{R}_1$ has a finite initial segment in $\mathcal{F}$,
and by minimality of $n(X)$ no member of $\mathcal{F}$  is a proper initial segment of another member of $\mathcal{F}$.

Steps (3) and (4) are the same for each topological Ramsey space.
Define
\begin{equation}
\mathcal{F}|X=\{a\in\mathcal{F}:\exists Y\le X,\ \exists n<\om\,  (a=r_n(Y))\}.
\end{equation}
Then let
$\mathcal{U}\re\mathcal{F}$ be the filter on the countable base set  $\mathcal{F} $
generated by $\{\mathcal{F}|X:X\in\mathcal{U}\}$.
By genericity of $\mathcal{U}$,
$\mathcal{U}\re\mathcal{F}$ is an ultrafilter.
The set-up to Step (4) is the same as in
Subsection  \ref{subsection.initialTRKstructure}.
Define $g:\mathcal{F}\ra \om$ by $g(a)=\min(\hat{f}(a))$, for each $a\in\mathcal{F}$.
By Theorem \ref{thm.50}, if $\mathcal{V}$ is nonprincipal then
the RK image of $\mathcal{U}\re\mathcal{F}$ under $g$,
$g(\mathcal{U}\re\mathcal{F})$, equals $\mathcal{V}$.

Now we want to understand  these maps $g$  so that we can understand the isomorphism class of $\mathcal{V}$ for Step (5).
$g$ induces an equivalence relation on the front $\mathcal{F}$.
As in the Ellentuck space, there is a notion of canonical equivalence relation for $\mathcal{R}_1$,
found in \cite{Dobrinen/Todorcevic14} and described now.
Since it is dense in $\mathcal{R}_1$ to find a $Z$ such that the equivalence relation $\mathcal{F}|Z$ is canonical, such a $Z$ will be in $\mathcal{U}$.
Assume then that $Z\in\mathcal{U}$ and $g$ is canonical on $\mathcal{F}|Z$.
We now describe these canonical equivalence relations proved in \cite{Dobrinen/Todorcevic14} beginning  with some simple examples, building up intuition for  the general case.

First we describe the canonical equivalence relations on the set $\mathcal{R}_1(n)$ as these are the building blocks for the canonical equivalence relations on fronts.

\begin{defn}[Canonical equivalence relations on $\mathcal{R}_1(n)$, \cite{Dobrinen/Todorcevic14}]
\label{defn.R(n)canonical}
Let $T$ be any subtree of $\bT(n)$.
Given $a(n)\in\mathcal{R}_1(n)$,
let $\pi_T(a(n))$ be the projection of the tree $a(n)$
to the nodes in the same position as $T$ inside $\bT(n)$.
That is, if $\iota:\bT(n)\ra a(n)$ is the tree isomorphism between them,
then $\pi_T(a(n))=\iota''T$.
The canonical equivalence relation $\E_T$ on $\mathcal{R}_1(n)$ is defined by
$a(n)\, \E_T\, b(n)$ if and only if  $\pi_T(a(n))=\pi_T(b(n))$.
\end{defn}

If $\mathcal{F}=\mathcal{AR}_1$, then
 the canonical equivalence relations  on $\mathcal{AR}_1$ are simply those given by
subtrees of $\bT(0)=\{\lgl\rgl,\lgl 0 \rgl,\lgl 0,0\rgl\}$.
Thus, there are three canonical equivalence relations on $\mathcal{AR}_1$.
By density, we may without loss of generality assume that $g$ induces a canonical equivalence relation on all of $\mathcal{AR}_1$; say this is given by $\pi_T$.
If $T=\{\lgl\rgl\}$, then
all members of $\mathcal{AR}_1$ are equivalent; this means that $g$ is a constant function.
In this case, $\mathcal{V}$  is a  principal ultrafilter generated by the singleton $g(a(0))$ for  each/every $a(0)\in\mathcal{AR}_1$.
If $T=\{\lgl \rgl, \lgl 0\rgl\}$, then
two members $a(0),b(0)\in
\mathcal{AR}_1$ are $\E_T$ equivalent
if and only if their level $1$ nodes are equal.
For example, if $a(0)=\{\lgl \rgl, \lgl 3\rgl,\lgl 3,0\rgl\}$ and $b(0)=\{\lgl \rgl, \lgl 3\rgl,\lgl 3,2\rgl\}$,
then they are equivalent,
but if  $a(0)=\{\lgl \rgl, \lgl 3\rgl,\lgl 3,0\rgl\}$ and $b(0)=\{\lgl \rgl, \lgl 4\rgl,\lgl 4,0\rgl\}$,
then they are not equivalent.
This means that $g$ is (up to  permutation) the projection map from the maximal nodes in $\bT$ to the the nodes in level $1$ of $\bT$.
This projection map yields the Ramsey ultrafilter $\mathcal{U}_0$.
Hence, in this case, $\mathcal{V}=g(\mathcal{U}\re\mathcal{AR}_1)$ which is the ultrafilter on base $\{\lgl n\rgl :n<\om\}$ generated by the projections of the members of $\mathcal{U}$ to their first levels.
This is $RK$ equivalent to
$\mathcal{U}_0$.
If $T=\{\lgl\rgl,\lgl 0\rgl, \lgl 0,0\rgl\}$,
then
 this means that $g$ is one-to-one on $\mathcal{AR}_1$.
Notice that $\mathcal{U}$ is isomorphic to the ultrafilter $\mathcal{U}\re\mathcal{AR}_1$.
Thus, if $g$ is one-to-one, then
$\mathcal{V}=g(\mathcal{U}\re\mathcal{AR}_1)
\equiv_{RK} \mathcal{U}_1$.

If $\mathcal{F}=\mathcal{AR}_2$,
then the possible canonical equivalence relations are given by two independent  canonical equivalence relations: on the $0$-th subtrees $\mathcal{R}_1(0)$ and on the first subtrees $\mathcal{R}_1(1)$.
Thus, there are five canonical equivalence relations on $\mathcal{R}_1(1)$.
Here, we start to see the more general pattern emerging.

\begin{fact}\label{fact.urcool}
For each $n$, the sets $\mathcal{R}_1(n)|X$, $X\in\mathcal{U}$, generate an ultrafilter on base set $\mathcal{R}_1(n)$.
Denote these as $\mathcal{U}\re \mathcal{R}_1(n)$.
\end{fact}

These are actually  p-points, as genericity of $\mathcal{U}$ ensures pseudointersections.

$g$ on $\mathcal{AR}_2$ is canonized by $(\E_{T_0},\E_{T_1})$, where $T_0$ is a subtree of $\bT(0)$ and $T_1$ is a subtree of $\bT(1)$.
With some work,
one can check that the sets $\{\pi_{T_0}(a(0)):a(0)\in\mathcal{R}_1|X\}$, $X\in\mathcal{U}_1$,  generate an ultrafilter on base set $\{\pi_{T_0}(a(0)):a(0)\in\mathcal{R}_1(0)\}$;
and the sets $\{\pi_{T_1}(a(1)):a(1)\in\mathcal{R}_1|X\}$, $X\in\mathcal{U}_1$, generate an ultrafilter on base set
$\{\pi_{T_1}(a(1)):a(1)\in\mathcal{R}_1(1)\}$.
Denote these by $\pi_{T_i}(\mathcal{U}_1\re\mathcal{R}_1(i))$, $i\in 2$.
For both $i\in 2$,
if $T=\{\lgl\rgl\}$,
then the ultrafilter  $\pi_T(\mathcal{U}_1\re \mathcal{R}_1(i))$ is  principal.
If $T=\{\lgl \rgl,\lgl i\rgl\}$,
then $\pi_T(\mathcal{U}_1\re \mathcal{R}_1(i))$ is
isomorphic to the Ramsey ultrafilter $\mathcal{U}_0$.
For $T_1$ equal to $\{\lgl \rgl, \lgl 1\rgl, \lgl 1,0\rgl\}$ or $\{\lgl \rgl, \lgl 1\rgl, \lgl 1,1\rgl\}$,
$\pi_{T_1}(\mathcal{U}_1\re \mathcal{R}_1(1))$ is
isomorphic to $\mathcal{U}_1$.
The new ultrafilter we now see is in the case of $T_1=\bT(1)$, in which case $\pi_T(\mathcal{U}_1\re \mathcal{R}_1(1))
=\mathcal{U}_1\re\mathcal{R}_1(1)$, which is
a p-point.

With some more work, one finds that
$\mathcal{V}$ is isomorphic to a Fubini product of two ultrafilters which are either p-points, Ramsey, or principal.
We write $\mathcal{U}*\mathcal{V}$ to denote the Fubini product $\lim_{n\ra\mathcal{U}}\mathcal{V}_n$ where for all $n$, $\mathcal{V}_n=\mathcal{V}$.

\begin{equation}
\mathcal{V}=g(\mathcal{U}_1\re\mathcal{AR}_2)
=\pi_{T_0}(\mathcal{U}_1\re\mathcal{R}_1(0))*\pi_{T_1}(\mathcal{U}_1\re\mathcal{R}_1(1)).
\end{equation}

In general, the canonical equivalence relations on $\mathcal{R}_1(n)$ are given by the subtrees of $\bT(n)$.
There are $2^{n+1}+1$ many of subtrees of $\bT(n)$, hence that many canonical equivalence relations on the $n$-th subtrees.
As in the cases for $\mathcal{R}_1(0)$ and $\mathcal{R}_1(1)$,
for any $n<\om$,
if the canonical equivalence relation on $\mathcal{R}_1(n)$ is given by $T_n=\{\lgl\rgl\}$, then the projection map induces a principal ultrafilter.
If $T_n=\{\lgl\rgl, \lgl n\rgl\}$,
then the projection map induces an ultrafilter isomorphic to the Ramsey ultrafilter $\mathcal{U}_0$.
If $T_n=\{\lgl\rgl,\lgl n\rgl\}\cup\{\lgl n,i\rgl:i\in I_n\}$ where $I_n$ is some nonempty subset of $n+1$,
then letting $k=|I_n|$,
the projection map induces an ultrafilter isomorphic to the p-point $\mathcal{U}_1\re\mathcal{R}_1(k-1)$.

Thus, if $\mathcal{F}$ is equal to $\mathcal{AR}_m$ for some $m$,
then there are trees $T_i\sse \bT(i)$ for each $i<m$ such that
\begin{equation}
\mathcal{V}\equiv_{RK}\pi_{T_0}(\mathcal{U}_1\re\mathcal{R}_1(0))*\cdots*
\pi_{T_{m-1}}(\mathcal{U}_1\re\mathcal{R}_1(m-1)).
\end{equation}
Each of these ultrafilters in the $m$-iterated Fubini product is either principal, $\mathcal{U}_0$, or $\mathcal{U}_0\re\mathcal{R}_1(k)$ for some $k$.

Now it may well be that the front $\mathcal{F}$ is not of the form $\mathcal{AR}_m$ for any $m$, but is more complex.
In this case, we still can do an analysis, using the complexity (uniform rank) of the front to conclude that $\mathcal{V}$ is isomorphic to a countable iteration of Fubini products where each ultrafilter  in the construction is a member of the collection
\begin{equation}
\mathcal{Y}_{\mathcal{R}_1}=\{1,\mathcal{U}_1\}\cup\{\mathcal{U}_1\re
\mathcal{R}_1(k):k<\om\},
\end{equation}
where $1$ denotes any principal ultrafilter.
This completely classifies the isomorphism types within the Tukey types of ultrafilters Tukey reducible to $\mathcal{U}_1$.

%%%%%%%%%%%%%%%%%%%%
%%%%%%%%%%%%%%%%%%%%
%%%%%%%%%%%%%%%%%%%%
%%%%%%%%%%%%%%%%%%%%

\section{Ultrafilters of Blass constructed by $n$-square forcing and extensions to hypercube forcings}

In the study of the structure of the Rudin-Keisler classes of p-points, Blass showed that not only can there be chains of order type $\om_1$ and $\mathbb{R}$,
but also that there can be Rudin-Keisler incomparable p-points.
In \cite{Blass73}, Blass
proved that assuming Martin's Axiom,
there is a p-point which has two Rudin-Keisler incomparable p-points RK below it.
We will call this forcing {\em $n$-square forcing}, $\bP_{\mathrm{square}}$, since a subset $X\sse \om\times\om$ is in the partial ordering $\bP_{\mathrm{square}}$ if and only if for each $n<\om$ there are sets $a,b$ each of size $n$ such that the product $a\times b$ is a subset of $X$.
$\bP$ is partially ordered by $\sse$.
The projections to the first  and second coordinates yield the two RK-incomparable p-points.
The p-point  obtained from this construction satisfies the partition relation
\begin{equation}
\mathcal{U}\ra(\mathcal{U})^2_{k,5}.
\end{equation}

The forcing $\bP_{\mathrm{square}}$ contains a dense subset which forms a topological Ramsey space denoted $\mathcal{H}^2$, which appears in
\cite{TrujilloThesis} and  \cite{Dobrinen/Mijares/Trujillo14}.
The members of this space are essentially infinite sequences which are a
 product of  two members of $\mathcal{R}_1$ in the following sense.
Let $\bT^2$ be the sequence of trees $\lgl \bT^2(n):n<\om\rgl$ such that
for each $n<\om$,
\begin{equation}
\bT^2(n)=\{\lgl\rgl,\lgl n\rgl\}\cup\{\lgl n,\lgl i,j\rgl\rgl:i,j\in n+1\}.
\end{equation}
$\bT^2(n)$ should  be thought of as a tree with height two where  levels $0$ and $1$ have one node, and level $2$ has an $(n+1)\times(n+1)$ square of nodes.
A sequence
$X=\lgl X(n):n<\om\rgl$
 is a member of $\mathcal{H}^2$ if and only if
it is a subtree of $\bT^2$ with the same structure as $\bT^2$.
Specifically, $X\in\mathcal{H}^2$ if and only if
 there is a strictly increasing sequence $(k_n)$ such that
 for each $n<\om$,
\begin{equation}
X(n)=\{\lgl\rgl, \lgl k_n\rgl\}\cup\{\lgl k_n,\lgl i,j\rgl\rgl:i\in I_n,\ j\in J_n\},
\end{equation}
 where $I_n, J_n\in [k_n+1]^{n+1}$.
We call $X(n)$ the {\em $n$-th block} of $X$.
For $X$ and $Y$ in $\mathcal{H}^2$,
$Y\le X$ if and only if for each $n$ there is a $k_n$ such that $Y(n)\sse X(k_n)$ and moreover, the sequence $(k_n)_{n<\om}$ is strictly increasing.
However, by the structure of the members of $\mathcal{H}^2$, it turns out that $\le$ is the same as $\sse$.

For the space $\mathcal{H}^2$,
the almost reduction $\le^*$
is simply $\sse^*$.
The forcing $(\mathcal{H}^2,\le^*)$ is $\sigma$-complete and  produces a new p-point, $\mathcal{U}_2$.
The RK structure below $\mathcal{U}_2$ is a diamond shape.   $\mathcal{U}_2$  has two  RK incomparable predecessors, namely the projections to the first and second directions.
These projected ultrafilters are actually generic for $(\mathcal{R}_1,\le_{\mathcal{R}_1}^*)$.
The projection to the nodes of length one produces a Ramsey ultrafilter.
Thus, we see that the structure of the RK classes reducible to  $\mathcal{U}_2$ includes the structure of the Boolean algebra $\mathcal{P}(2)$.
In fact,
Ramsey-theory techniques along with the canonical equivalence relations, similar  to those in the previous section,
allow us to deduce that these are the only RK types of ultrafilters RK reducible to $\mathcal{U}_2$.

Similarly to the space $\mathcal{R}_1$,
for each $n$, there is an ultrafilter $\mathcal{U}_2\re\mathcal{H}^2(n)$ which is the ultrafilter on base set $\mathcal{H}^2(n)$ generated by the sets
\begin{equation}
\mathcal{H}^2(n)|X:=\{a(n):\exists Y\le X\, (a=r_{n+1}(Y))\}.
\end{equation}
Each of these ultrafilters is a p-point.
We point out that the ultrafilter $\mathcal{U}_2$ is isomorphic to the ultrafilter $\mathcal{U}_2\re\mathcal{H}^2(0)$.

The canonical equivalence relations on the $n$-th blocks $\mathcal{H}^2(n)=\{X(n):X\in\mathcal{H}^2\}$
are given  by canonical projections of the following forms.
Recall the \Erdos-Rado canonical projections on finite sets of natural numbers:
Given $I\sse n+1$,
for any $c=\{c_0,\dots,c_n\}$,
$\pi_I(c)=\{c_i:i\in I\}$.
Suppose $a(n)$ is a member of $\mathcal{H}^2(n)$ and
$a(n)=\{\lgl\rgl, \lgl k\rgl\}\cup\{\lgl k,\lgl i,j\rgl\rgl:i\in I_a, j\in J_a \}$.
Given  $T_0, T_1$  subtrees of  $\bT(n)$ (the $n$-th block of the tree $\bT$ from $\mathcal{R}_1$)
define
\begin{equation}\label{eq.31}
\pi_{T_0,T_1}(a(n))= \{\lgl\rgl\}\mathrm{\ if\ } T_0=T_1=\{\lgl\rgl\}
\end{equation}
\begin{equation}\label{eq32}
\pi_{T_0,T_1}(a(n))= \{\lgl \rgl, \lgl k\rgl\}\mathrm{\ if \ } T_0=T_1=\{\lgl\rgl,\lgl k\rgl\}
\end{equation}
\begin{align}\label{eq.33}
\pi_{T_0,T_1}(a(n))= &\{\lgl \rgl,\lgl k\rgl\}\cup\{ \lgl k,i\rgl:i\in \pi_{I_0}(I_a)\}\cr
& \ \ \mathrm{\ if \ }
T_0=\{\lgl\rgl,\lgl k\rgl\}\cup\{ \lgl k,i\rgl:i\in I_0\}
\mathrm{\ and\ }
T_1=\{\lgl\rgl,\lgl k\rgl\}
\end{align}
\begin{align}\label{eq.34}
\pi_{T_0,T_1}(a(n))= & \{\lgl \rgl,\lgl k\rgl\}\cup\{ \lgl k,j\rgl:j\in \pi_{I_1}(J_a)\}\cr
& \ \ \mathrm{\ if \ }T_0=\{\lgl\rgl,\lgl k\rgl\}
 \ \ \mathrm{\ and\ }
T_1=\{\lgl\rgl,\lgl k\rgl, \lgl k,i\rgl:i\in I_1\}
\end{align}
\begin{equation}\label{eq.35}
\pi_{T_0,T_1}(a(n))= \{\lgl \rgl,\lgl k\rgl\}\cup\{ \lgl k,\lgl i,j\rgl\rgl:i\in\pi_{I_0}(I_a),\ j\in \pi_{I_1}(J_a)\}.
\end{equation}
An equivalence relation $\E$ on $\mathcal{H}^2(n)$ is {\em canonical} if
 for each $i\in\{0,1\}$, there is some tree $T_i\sse \bT(n)$ such that
for $a(n),b(n)\in\mathcal{H}^2(n)$,
\begin{equation}\label{eq.defcanon}
a(n)\, \E\,  b(n)
\Longleftrightarrow
\pi_{T_0,T_1}(a(n))=\pi_{T_0,T_1}(b(n)).
\end{equation}

The {\em initial Rudin-Keisler structure} below $\mathcal{U}_2$ is the collection of all isomorphism types of nonprincipal ultrafilters RK reducible to $\mathcal{U}_2$.
This turns out to be exactly the shape of $(\mathcal{P}(2),\sse)$, that is, a diamond shape.
Since $\mathcal{U}_2$ is isomorphic to $\mathcal{U}_2\re\mathcal{H}^2(0)$,
if
 $\mathcal{V}\le_{RK}\mathcal{U}_2$,
then there is a map $h$ from $\mathcal{H}^2(0)$ into $\om$ such that
$h(\mathcal{U}_2\re\mathcal{H}^2(0))
=\mathcal{V}$.
Without loss of generality, we may assume that $h$ is canonical, represented by some projection map $\pi_{T_0,T_1}$.
If both of $T_0,T_1$ are $\{\lgl\rgl\}$, then the RK image of $\mathcal{U}_2$ is a principal ultrafilter.
If  both are  $\{\lgl\rgl,\lgl 0\rgl\}$  then
 the $h$-image of $\mathcal{U}_2$ is a  Ramsey ultrafilter.
If $T_0=\{\lgl\rgl,\lgl 0\rgl,\lgl 0,0\rgl\}$
and $T_1=\{\lgl\rgl,\lgl 0\rgl\}$  then
 the $h$-image of $\mathcal{U}_2$ is isomorphic to the ultrafilter forced by $(\mathcal{R}_1,\le^*_{\mathcal{R}_1})$ and hence is weakly Ramsey; denote this as $\mathcal{V}_0$.
Likewise if
$T_1=\{\lgl\rgl,\lgl 0\rgl,\lgl 0,0\rgl\}$
and $T_0=\{\lgl\rgl,\lgl 0\rgl\}$  then
 the $h$-image of $\mathcal{U}_2$
 $\mathcal{V}_1$.
Lastly, if both $T_0=T_1=\{\lgl\rgl,\lgl 0\rgl,\lgl 0,0\rgl\}$
the  $h$-image is isomorphic to $\mathcal{U}_2$.
Thus, we find the exact structure of the RK types below $\mathcal{U}_2$.
We call this an initial RK structure, since it is downwards closed in the RK classes.

The following ultrafilters form the building blocks
for understanding
 the Tukey types of ultrafilters Tukey reducible to $\mathcal{U}_2$.
The canonical projections applied to the p-point $\mathcal{U}_2\re\mathcal{H}^2(n)$
yield the following p-points.
Let $T_0$ and $T_1$ be subtrees of $\bT(n)$.
If $T_0=T_1=\{\lgl\rgl\}$,
then $\pi_{T_0,T_1}(\mathcal{U}_2\re\mathcal{H}^2(n))$ is simply a principal ultrafilter.
If $T_0=T_1=\{\lgl\rgl,\lgl n\rgl\}$,
then $\pi_{T_0,T_1}(\mathcal{U}_2\re\mathcal{H}^2(n))$ is
isomorphic to the projected Ramsey ultrafilter $\mathcal{U}_0$, similarly to the space $\mathcal{R}_1$.
If $T_0$ and $T_1$ are as in Equation (\ref{eq.33}),
then
 $\pi_{T_0,T_1}(\mathcal{U}_2\re\mathcal{H}^2(n))$
is isomorphic to $\mathcal{U}_1\re\mathcal{R}_1(l)$, where $l=|I_0|$.
Likewise, if
 $T_0$ and $T_1$ are as in Equation (\ref{eq.34}),
then $\pi_{T_0,T_1}(\mathcal{U}_2\re\mathcal{H}^2(n))$
is isomorphic to $\mathcal{U}_1\re\mathcal{R}_1(l)$, where $l=|I_1|$.
If $T_0$ and $T_1$ are as in Equation (\ref{eq.35}),
then  $\pi_{T_0,T_1}(\mathcal{U}_2\re\mathcal{H}^2(n))$
is  a new type of p-point which has as base set the collection $\{\pi_{T_0,T_1}(a(n)):a(n)\in\mathcal{H}^2(n)\}$
and is generated by the sets
$\pi_{T_0,T_1}(\mathcal{H}^2(n)|X):=\{\pi_{T_0,T_1}(a(n)):\exists Y\le X\, (a=r_{n+1}(Y))\}$, $X\in\mathcal{U}_2$.
These are finite  trees of height two which have $|I_0|\times|I_1|$ rectangles as their maximal nodes.
Let $\mathcal{Y}^2(n)$ denote the collection of all these ultrafilters obtained by canonical projections on $\mathcal{H}^2(n)$.
Note that $\mathcal{Y}^2(n)$ is finite.

The Tukey types are handled similarly as for $\mathcal{R}_1$.
Each  monotone cofinal map from $\mathcal{U}_2$ to an ultrafilter $\mathcal{V}$ on $\om$
is continuous when restricted  below some member of $\mathcal{U}_2$.
As in the case of $\mathcal{R}_1$,
there is some front $\mathcal{F}$ and a function $g:\mathcal{F}\ra\om$  such that
$\mathcal{V}= g(\mathcal{U}_2\re \mathcal{F})$.
Again, $g$ may be assumed to be canonical, either by a forcing argument or construction some extra hypothesis like CH or less.
For fronts $\mathcal{F}$ of the form $\mathcal{AR}_m$,
$g$ is canonized by a sequence $(\pi_n:n\le m)$ of canonical projection maps so that
$g(\mathcal{U}_2\re \mathcal{F})$ is Rudin-Keisler equivalent to the Fubini iteration
\begin{equation}
\pi_0(\mathcal{U}_2\re\mathcal{H}^2(0))*\dots*
\pi_{m-1}(\mathcal{U}_2\re\mathcal{H}^2(m))
\end{equation}
where $\pi_n$ is one of the canonical projection maps on $\mathcal{H}^2(n)$.

\begin{thm}[Dobrinen/Trujillo, \cite{TrujilloThesis}]\label{thm.H^2}
Let $\mathcal{U}_2$ be the ultrafilter forced by $(\mathcal{H}^2,\le^*)$.
Then
\begin{enumerate}
\item
If $\mathcal{W}\le_{RK}\mathcal{U}_2$,
then $\mathcal{W}$ is isomorphic to one of $\mathcal{U}_2$, $\mathcal{V}_0$, $\mathcal{V}_1$, the projected Ramsey ultrafilter, or a principal ultrafilter.
Thus, the initial RK structure of the nonprincipal ultrafilters RK reducible to  $\mathcal{U}_2$ is simply the structure of the Boolean algebra $\mathcal{P}(2)$.
\item
If  $\mathcal{W}\le_T\mathcal{U}_2$,
then $\mathcal{W}$ is Tukey equivalent  to one of $\mathcal{U}_2$, $\mathcal{V}_0$, $\mathcal{V}_1$, the projected Ramsey ultrafilter, or a principal ultrafilter.
Thus, the initial Tukey  structure of the nonprincipal ultrafilters Tukey reducible to  $\mathcal{U}_2$ is simply the structure of the Boolean algebra $\mathcal{P}(2)$.
\item
If $\mathcal{W}\le_T\mathcal{U}_2$,
then the Rudin-Keisler classes inside of $\mathcal{W}$ are exactly those of the countable iterates of Fubini products of ultrafilters from the countable collection $\bigcup_{n<\om}\mathcal{Y}^2(n)$.
\end{enumerate}
\end{thm}

Continuing this  to higher dimensions,
the
 {\em hypercube} topological Ramsey spaces $\mathcal{H}^k$
have as elements infinite sequences $X=\lgl X(n):n<\om\rgl$ such that the $n$-th block  $X(n)$
consists of $\{\lgl\rgl,\lgl k_n\rgl\}$, where $(k_n)$ is a strictly increasing sequence, along with a $k$-dimensional cube
 which is the product $I_{k,i}$, where $I_{k,i}\in [k_n+1]^k$ for each $i<k$.
The partial ordering on the spaces $\mathcal{H}^k$ are  analogous to $\mathcal{H}^2$.
These spaces are shown by Trujillo and the author to be topological Ramsey spaces and appear in  \cite{Dobrinen/Mijares/Trujillo14}.
There,  a larger countable collection of p-points $\mathcal{Y}^k$ is found, and the analogous results to Theorem \ref{thm.H^2} for $k$-dimensions are proved.
Thus, the Boolean algebra $\mathcal{P}(k)$ is shown to be both an initial RK structure as well as an initial Tukey structure in the p-points.
The analogue of (3) in the thorem also holds, where the iterated Fubini products  range over ultrafilters in $\mathcal{Y}^k$.

%%%%%%%%%%%%%%
%%%%%%%%%%%%%%
%%%%%%%%%%%%%%
%%%%%%%%%%%%%%

\section{More initial Rudin-Keisler and Tukey structures obtained from topological Ramsey spaces}

The previous two sections provided details of how  the Rudin-Keisler and Tukey structures below certain p-points can be completely understood if the p-points were forced by some topological Ramsey space in which canonical equivalence relations on fronts are well-understood.
This section  gives the reader the flavor of a collection of  broader results.

\subsection{$k$-arrow, not $(k+1)$-arrow ultrafilters}\label{subsection.arrow}

The $k$-arrow ultrafilters are a class of p-points which satisfy asymmetric partition relations.

\begin{defn}[\cite{Baumgartner/Taylor78}]\label{defn.arrow}
An ultrafilter $\mathcal{U}$ is {\em $n$-arrow} if $2\le n<\om$ and for every function $f:[\om]^2\ra 2$,
either there exists a set $X\in\mathcal{U}$ such that $f([X]^2)=\{0\}$,
or else there is a set $Y\in [\om]^n$ such that $f([Y]^2)=\{1\}$.
$\mathcal{U}$ is an {\em arrow} ultrafilter if $\mathcal{U}$ is $n$-arrow for each $n\le 3<\om$.
\end{defn}

Baumgartner and Taylor showed in  \cite{Baumgartner/Taylor78} that for each $2\le n<\om$, there are p-points which are $n$-arrow but not $(n+1)$-arrow.
Note that every ultrafilter is $2$-arrow.
Similarly to the $\mathcal{R}_1$ and $\mathcal{H}^2$ spaces, for each $k\ge 2$, there is a topological Ramsey space  $\mathcal{A}_k$ which  is dense in the forcing that Baumgartner and Taylor used to construct an $k$-arrow, not $(k+1)$-arrow ultrafilter.
The members of this space are infinite sequences, $X=\lgl X(n):n<\om\rgl$, such that each $X(n)$ is a certain type of ordered graph omitting $k+1$-cliques.
The fact that these ultrafilters are forced by a topological Ramsey space shows that they have complete combinatorics, by Theorem \ref{thm.completecomb}.
For details, the reader is referred to \cite{Dobrinen/Mijares/Trujillo14}.

Similarly to $\mathcal{R}_1$,
both the initial RK structure and initial Tukey structure for the $k$-arrow, not $k+1$-arrow ultrafilter  $\mathcal{W}_k$ forced by $(\mathcal{A}_k,\le^*)$
are of size $2$:  $\mathcal{W}_k$ and its projection to a Ramsey ultrafilter.
However, when we look at the Rudin-Keisler classes inside of the Tukey type of $\mathcal{W}_k$,
the picture becomes more  complex as we shall now see.

The canonical equivalence relations on the collection of $n$-th blocks (that is, $\{X(n):X\in\mathcal{A}_n\}$)
were obtained by the author and we found to be again given by projections.
This depended heavily on the flexibility of the structure of the \Fraisse\ limit  of the class of finite ordered graphs omitting $k+1$-cliques.
The following is a specific case of a more general theorem for canonical equivalence relations, attributed to Dobrinen in \cite{Dobrinen/Mijares/Trujillo14}.
For graphs $A,B$, the notation ${B\choose A}$ denotes the set of all subgraphs of $B$ which are isomorphic to $A$.
For an ordered graph $A$ with vertices $\{v_0,\dots,v_j\}$ and $I\sse j+1$,
$\pi_I(A)$ denotes the subgraph of $A$ induced by the vertices $\{v_i:i\in I\}$.

\begin{thm}[Dobrinen, \cite{Dobrinen/Mijares/Trujillo14}]
\label{thm.ERgraphs}
Let $k\ge 3$ be given and let $A$ and $B$ be finite ordered graphs  omitting $k$-cliques and such that $A$ embeds into $B$ as a subgraph.
Then there is a finite ordered graph $C$ omitting $k$-cliques which is large enough that the following  holds.
Given any equivalence relation $\E$ on ${ C \choose A}$,
there is an $I\sse |A|$ and a $B'\in {C\choose B}$ such that
$\E$ restricted to ${B' \choose A}$ is given by $\E_I$.
\end{thm}

The building blocks of the Rudin-Keisler classes inside the Tukey  are ultrafilters obtained by the canonical projection maps resulting in the following.
Let $\mathcal{K}_{k+1}$ denote the \Fraisse\ class of all finite ordered graphs omitting $k+1$-cliques.
More precisely, we take one finite ordered graph omitting $k+1$-cliques from each isomorphism class of these graphs.
This set is partially ordered by graph embedding.

\begin{thm}[Dobrinen/Mijares/Trujillo, \cite{Dobrinen/Mijares/Trujillo14}]
\label{thm.arrow}
Let $\mathcal{W}_k$ be a $k$-arrow, not $k+1$-arrow p-point forced by the topological Ramsey space $\mathcal{A}_k$ partially ordered by the $\sigma$-closed order $\le ^*$.
\begin{enumerate}
\item
The initial Rudin-Keisler structure below $\mathcal{W}_k$ is a chain of length $2$.
\item
The initial Tukey structure below $\mathcal{W}_k$ is  a chain of length $2$.
\item
The isomorphism classes  inside the Tukey type of  $\mathcal{W}_k$ have the same structure as $\mathcal{K}_{k+1}$ partially ordered by embedding.
\end{enumerate}
\end{thm}

This is a particular example of more general results in \cite{Dobrinen/Mijares/Trujillo14}
handling other \Fraisse\ classes of ordered relational structures with the Ramsey property, and finite  products of such structures, producing quite complex Rudin-Keisler structures inside the Tukey types.
For instance, there are topological Ramsey spaces which produce initial Tukey structure  of the form $([\om]^{<\om},\sse)$.
Furthermore, {\em arrow ultrafilters} (ultrafilters which are $k$-arrow for all $k<\om$)  are also seen to be forced by  a topological Ramsey space, and have similar results for their initial RK and Tukey structures.
These and more  general results are  found as
Theorems 60 and  67 in \cite{Dobrinen/Mijares/Trujillo14}.

%%%%%%%%%%%%%%%%%%%%%%5
%%%%%%%%%%%%%%%%%%%%%%5

\subsection{Ultrafilters of Laflamme with increasingly weak partition relations}

The ultrafilter $\mathcal{U}_1$ in Section \ref{section.R_1}
was only the beginning of a hierarchy of p-points satisfying successively weaker partition relations constructed by
 Laflamme in \cite{Laflamme89}.
These forcings $\bP_{\al}$, $1\le \al<\om_1$,
were found to have dense subsets forming topological Ramsey spaces
in \cite{Dobrinen/Todorcevic15}.
The reader interested in more details is referred to that paper.
Here, we merely state that this yields rapid  p-points $\mathcal{V}_{\al}$ for each $1\le \al<\om_1$ which have complete combinatorics (proved by Laflamme for the version over HOD$^{V[G]}$, and obtained over $L(\mathbb{R})$ in the presence of large cardinals by Dobrinen and Todorcevic by virtue of being forced by a topological Ramsey space).

\begin{thm}[Dobrinen/Todorcevic, \cite{Dobrinen/Todorcevic15}]\label{thm.D15}
For each $1\le \al<\om_1$, there is a topological Ramsey space $\mathcal{R}_{\al}$ forcing
a p-point $\mathcal{V}_{\al}$ such that the initial Rudin-Keisler structure and the initial Tukey structure  are both decreasing chains of order-type $(\al+1)^*$.

For each $1\le \al<\om_1$,  the Rudin-Keisler types inside the Tukey type of $\mathcal{V}_{\al}$
are the countable Fubini iterates of the p-points obtained by canonical projections on the blocks of the sequences forming members of $\mathcal{R}_{\al}$.
\end{thm}

Recent work of Zheng in \cite{Zheng17} showed that these ultrafilters are preserved by countable support side-by-side Sacks forcing.
Zheng had already shown this to be the case for the ultrafilter on base set FIN = $[\om]^{<\om}\setminus\{\emptyset\}$ which is constructed by the Milliken space of infinite increasing block sequences (see \cite{Zheng16}).

\subsection{Ultrafilters forced by $\mathcal{P}(\om\times\om)/(\Fin\otimes \Fin)$}\label{subsec.finfin}

The forcing $\mathcal{P}(\om)/\Fin$ which adds a Ramsey ultrafilter has a natural generalization to $\mathcal{P}(\om\times\om)/\Fin\otimes\Fin$,
where $\Fin\otimes\Fin$ is the ideal of the sets $X\sse\om\times\om$ such that for all but finitely many $i<\om$, the set $\{j<\om:(i,j)\in X\}$ is finite.
We let $(X)_i$ denote $\{j<\om:(i,j)\in X\}$ and call it the {\em $i$-th fiber} of $X$.
This forcing adds  a new ultrafilter $\mathcal{W}_2$ which is not a p-point but
satisfies the best partition property that a non-p-point can have, namely,
$\mathcal{W}_2\ra(\mathcal{W}_2)^2_{l,4}$.
Letting $\pi_0:\om\times\om\ra\om$ by $\pi_0(i,j)=i$,
the projection $\pi_0(\mathcal{W}_2)$ to its first coordinates is a Ramsey ultrafilter.

Many properties of the ultrafilter $\mathcal{W}_2$  were investigated by Blass, Dobrinen and Raghavan in \cite{Blass/Dobrinen/Raghavan15}.
That paper included bounds on  Tukey type of $\mathcal{W}_2$ showing that it is neither minimal nor maximal in the Tukey types of ultrafilters, but the question of the exact structure of the Tukey types below it remained open.

In \cite{DobrinenJSL15}, the author proved that
$\mathcal{P}(\om\times\om)/\Fin\otimes\Fin$ is forcing equivalent to a topological Ramsey space when partially ordered by its almost reduction relation.
The coideal $(\Fin\otimes\Fin)^+$
is the collection of all $X\sse\om\times\om$ such that for all
but finitely many $i<\om$,
the $i$-th fiber of $X$ is infinite.
For $X,Y\sse\om\times\om$,
write $Y\sse^{*2} X$ if and only if $Y\setminus X\in\Fin\otimes\Fin$.
It is routine to check that $\mathcal{P}(\om\times\om)/\Fin\otimes\Fin$ is forcing equivalent to $((\Fin\otimes\Fin)^+, \sse^{*2})$.
The
forcing $((\Fin\otimes\Fin)^+, \sse^{*2})$ contains a dense subset which forms a topological Ramsey space.
We denote this space $\mathcal{E}_2$, since it is the 2-dimensional Ellentuck space.
Here, we will only present an overview of this work, referring the interested reader to \cite{DobrinenJSL15}.

In order to find the initial Rudin-Keisler and Tukey structures below $\mathcal{W}_2$,
a new kind of canonical Ramsey theorem for equivalence relations on fronts had
to be proved.   The canonical equivalence relations are again given by canonical projection functions, projecting to subtrees.  However, they  have a quite different structure than the previous examples in that they are not sequences of finitary projections, since the structure of the members of $\mathcal{E}_2$ are isomorphic to the ordinal $\om^2$.

Similarly to how $\Fin^{\otimes 2}:=\Fin\otimes\Fin$ was defined given $\Fin$,
the process can be recursively continued to define  ideals
 $\Fin^{\otimes k+1}$ for each $k<\om$,
where $X\sse \om^{k+1}$ is a member of
 $\Fin^{\otimes k+1}$ if and only if for all but finitely many $i_0<\om$,
$\{(i_1,\dots,i_k)\in\om^k:(i_0,i_1,\dots,i_k)\in X\}$ is a member of $\Fin^{\otimes k}$.
Then $\mathcal{P}(\om^k)/\Fin^{\otimes k}$ forces an ultrafilter $\mathcal{W}_k$ which is not a p-point and projects to $\mathcal{W}_j$ for each $1\le j<k$, where $\mathcal{W}_1$ is Ramsey.
The initial Rudin-Keisler and Tukey structures of these ultafilters are as follows.

\begin{thm}[Dobrinen, \cite{DobrinenJSL15}]\label{thm.jsl}
For each $2\le k<\om$, there is a $k$-dimensional Ellentuck space,  $\mathcal{E}_k$, such that $(\mathcal{E}_k,\sse^{*k})$ is forcing equivalent to  $\mathcal{P}(\om^k)/\Fin^{\otimes k}$.
The forced
ultrafilter $\mathcal{U}_k$ satisfies complete combinatorics,
and its initial Rudin-Keisler and its initial Tukey structures are both chains of length $k$.
\end{thm}

Although these ultrafilters $\mathcal{W}_k$ are not p-points, the high-dimensional Ellentuck spaces which force them treat them as p-points,
in the sense that every $\sse^{*k}$-decreasing sequence has a $\sse^{*k}$-pseudointersection in $\mathcal{W}_k$.
This is the sense in which these ultrafilters are similar to p-points; they satisfy diagonalization with respect to some $\sigma$-closed ideal.
It is efficacious to think of these as p-points with respect to topological Ramsey spaces  with respect to  almost reduction.

%%%%%%%%%%%%%%%%%%%%%%%%
%%%%%%%%%%%%%%%%%%%%%%%%
%%%%%%%%%%%%%%%%%%%%%%%%
%%%%%%%%%%%%%%%%%%%%%%%%

\section{Further Directions}\label{sec.}

The construction of topological Ramsey spaces to
has served to fine-tune
our understanding of several classes of  ultrafilters satisfying some partition relations.
It seems to us that these are just a few examples of a broader scheme.
Listed below are some guiding themes for further investigation in which topological Ramsey spaces  will likely  play a vital role.
Recall our Conjecture \ref{conj}:
Every ultrafilter satisfying some partition relation and forced by some $\sigma$-closed forcing
is actually forced by some topological Ramsey space.
If this turns out to be true, then topological Ramsey spaces will be exactly the correct spaces in which to investigate such ultrafilters, and moreover, all such ultrafilters will have complete combinatorics.

Finding the initial Tukey structures is a way of approximating the exact structure of {\em all} the Tukey types of ultrafilters starting  from the bottom of the hierarchy and going as high up as possible.
We have shown  that $(\al+1)^*$ for each $1\le\al<\om_1$,
$\mathcal{P}(k)$ for each $1\le k<\om$,
and $([\om]^{<\om},\sse)$ all appear as initial Tukey structures of p-points.
We have also shown that each finite chain of length two or more appears as an initial Tukey structure of a non-p-point.
Furthermore, \cite{DobrinenJML16} and a forthcoming paper obtain uncountable linear orders as initial Tukey structures of non-p-points.
In \cite{Dobrinen/Todorcevic11}, Dobrinen and Todorcevic constructed $2^{\mathfrak{c}}$ many Tukey incomparable Ramsey ultrafilters assuming cov$(\mathcal{M})=\mathfrak{c}$,
showing that this large antichain appears as an initial Tukey structure.
This is in contrast to  other work in
\cite{Milovich08}, \cite{Dobrinen/Todorcevic11}, \cite{Raghavan/Todorcevic12} and \cite{Raghavan/Shelah17} showing that certain structures embed into the Tukey types of ultrafilters.
We would like to know the structure of  downward closed Tukey structures which are as large as possible as a means of gaining information about the exact structure of all Tukey types of ultrafilters.

\begin{problem}\label{prob.1}
Given an ultrafilter $\mathcal{U}$ satisfying some partition property, what is the structure of the Tukey types of all ultrafilters Tukey reducible to $\mathcal{U}$?
\end{problem}

For the examples
analyzed in previous sections,
knowing that the generating partial ordering is essentially a topological Ramsey space aided greatly in solving this problem.
The Ramsey theory available also enabled us to find initial Rudin-Keisler structures and precisely, the structure of the RK classes inside the Tukey types of an initial structure of Tukey types.
Similar questions can be asked for these two foci.

A related but more challenging problem is the following.

\begin{problem}\label{prob.2}
What are the most complex structures initial Tukey structures that can be found?
\end{problem}

If one can find the initial Tukey structure below the maximal Tukey type, then one has completely found the full  Tukey structure of all ultrafilters.
It should be pointed out that
it may be consistent that there is only one Tukey type.  This is what remains of Isbell's Problem in \cite{Isbell65} which is  one of the most important questions on Tukey types of ultrafilters.
Such a model would have to contain no p-points and hence no Ramsey ultrafilters, so it was not the focus of this paper.

As we briefly saw in Subsection \ref{subsec.finfin}, forcing with the $\sigma$-closed forcings $\mathcal{P}(\om^k)/\Fin^{\otimes k}$,  and  more generally
$\mathcal{P}(\om^{\al})/\Fin^{\otimes\al}$ for each countable ordinal $\al$ (see  \cite{DobrinenJML16}),
and for the other examples covered in previous sections,
forcing with some partial ordering modulo a $\sigma$-ideal  produces an ultrafilter which has complete combinatorics, since, for these examples, they are forced by topological Ramsey spaces.
This leads to the following question which we find quite compelling.

\begin{problem}
Given a countable set $X$ and a $\sigma$-closed ideal $\mathcal{I}$ on $X$,
if the forcing $\mathcal{P}(X)/\mathcal{I}$ adds an ultrafilter which satisfies some weak partition properties,
is there some topological Ramsey space $\mathcal{R}$ such that $(\mathcal{R},\le^*)$ is forcing equivalent to  $\mathcal{P}(X)/\mathcal{I}$?
\end{problem}

A related question is the following.

\begin{problem}\label{problem.50}
For which
$\sigma$-closed ideals $\mathcal{I}$ on a countable set $X$, such that
 the forcing $\mathcal{P}(X)/\mathcal{I}$ adds an ultrafilter which
  has complete combinatorics,
is there some topological Ramsey space $\mathcal{R}$ such that $(\mathcal{R},\le^*)$ is forcing equivalent to  $\mathcal{P}(X)/\mathcal{I}$?
\end{problem}

In  \cite{Hrusak/Verner11}, Hru\v{s}ak and Verner proved  that if $\mathcal{I}$ is a tall $F_{\sigma}$ P-ideal, then $\mathcal{P}(\om)/\mathcal{I}$ adds a p-point which has no rapid RK-predessor and which is not Canjar.  Thus, there is no Ramsey ultrafilter RK below this forced ultrafilter, but the Mathias forcing with tails in this ultrafilter does add a dominating real.
It seems unlikely  that such ideals give an affirmative answer to Problem \ref{problem.50}
since all know topological Ramsey spaces have ultrafilters with Ramsey ultrafilters RK below them,
but this remains open.

Lastly, we would like to have a more user-friendly characterization of complete combinatorics for topological Ramsey spaces.
The characterization  of complete combinatorics  given by Di Prisco, Mijares, and Nieto
in Theorem \ref{thm.completecomb}
requires one to understand selectivity of an ultrafilter  in the sense of diagonalizations of certain
dense open sets  with respect to the ultrafilter.
The complete combinatorics of Blass and Laflamme,  on the other hand,
characterize complete combinatorics in terms of Ramsey properties.
We conjecture that a similar characterization can be given for the topological Ramsey space setting.
We say that a filter $\mathcal{U}$ on base set $\mathcal{R}$  is {\em Ramsey}  with respect to $(\mathcal{R},\le,r)$ if and only if
for each $2\le n<\om$,
for each coloring $c:\mathcal{AR}_n\ra 2$, there is an $X\in\mathcal{U}$ such that $c$ has one color on $\mathcal{AR}_n|X$.

\begin{conjecture2}\label{prob.4}
Let $(\mathcal{R},\le,r)$ be a topological Ramsey space, and let $\mathcal{U}$ be a filter on base set  $\mathcal{R}$.
Suppose that  there is a supercompact cardinal in $V$.
If $\mathcal{U}$ is Ramsey, then $\mathcal{U}$ is generic for the forcing $(\mathcal{R},\le^*)$ over the Solovay model $L(\mathbb{R})$.
\end{conjecture2}

It would suffice to prove that $\mathcal{U}$ is selective (in the sense of Definition \ref{defn.selective}) if and only if $\mathcal{U}$ is Ramsey.
This seems likely, as  similar (but not exactly the same) results  were obtained
for the ultrafilters in \cite{Dobrinen/Mijares/Trujillo14} and
for a class of topological Ramsey spaces of trees in \cite{TrujilloThesis}.
Such a representation of complete combinatorics over $L(\mathbb{R})$  for topological Ramsey spaces would be the ideal analogue of the
complete combinatorics of Blass and Laflamme.

%\end{document}

\bibliographystyle{amsplain}
\bibliography{references}

\end{document}